\documentclass{amsart}
\usepackage{amsmath,amssymb,amsthm,amsfonts}
\usepackage{mathrsfs}
\usepackage{amscd} 

\usepackage{bm}
\usepackage{hyperref}
\usepackage[normalem]{ulem}
\usepackage[T1]{fontenc}
\usepackage{lmodern}      
\usepackage[english]{babel}
\usepackage{microtype}               
\usepackage{graphicx}
\usepackage{enumitem}
\usepackage[curve]{xypic}
\usepackage[usenames,dvipsnames,x11names]{xcolor}
\usepackage{tikz}
\usepackage[all]{xy}
\usepackage{tikz-cd}
\usetikzlibrary{matrix}
\usetikzlibrary{arrows,decorations,shapes,shadows}
\usepackage{verbatim}
\usepackage{stmaryrd}
\usepackage{ocgx}

\newbox\mybox
\def\overtag#1#2#3{\setbox\mybox\hbox{$#1$}\hbox to
  0pt{\vbox to 0pt{\vglue-#3\vglue-\ht\mybox\hbox to \wd\mybox
      {\hss$\ss#2$\hss}\vss}\hss}\box\mybox}
\def\undertag#1#2#3{\setbox\mybox\hbox{$#1$}\hbox to 0pt{\vbox to
    0pt{\vglue#3\vglue\ht\mybox\hbox to \wd\mybox
      {\hss$\ss#2$\hss}\vss}\hss}\box\mybox}
\def\lefttag#1#2#3{\hbox to 0pt{\vbox to 0pt{\vglue -6pt\hbox to
      0pt{\hss$\ss#2$\hskip#3}\vss}}#1}
\def\righttag#1#2#3{\hbox to 0pt{\vbox to 0pt{\vglue -6pt\hbox to
      0pt{\hskip#3$\ss#2$\hss}\vss}}#1}
\let\ss\scriptstyle
\def\splicediag#1#2{\xymatrix@R=#1pt@C=#2pt@M=0pt@W=0pt@H=0pt}
\def\Dot{\lower.2pc\hbox to 2pt{\hss$\bullet$\hss}}
\def\Circ{\lower.2pc\hbox to 2pt{\hss$\circ$\hss}}
\def\Vdots{\raise5pt\hbox{$\vdots$}}
\newcommand\lineto{\ar@{-}}
\newcommand\dashto{\ar@{--}}
\newcommand\dotto{\ar@{.}}

\usepackage{xcolor}

\newcommand{\fract}[2]{\hbox{\leavevmode
  \kern.1em \raise .25ex \hbox{\the\scriptfont0 $#1$}\kern-.1em }\big/
  {\hbox{\kern-.15em \lower .5ex \hbox{\the\scriptfont0 $#2$}} }}

\renewcommand{\setminus}{\smallsetminus}
\newcommand\Q{{\mathbb Q}}
\newcommand\R{{\mathbb R}}
\newcommand\C{{\mathbb C}}
\newcommand\Z{{\mathbb Z}}
\newcommand{\LL}{\mathbb L}
\newcommand\N{{\mathbb N}}
\newcommand\pa{{\mathfrak a}}

\newcommand\pb{{\mathfrak b}}

\newcommand\pc{{\mathfrak c}}

\renewcommand{\t}{\tau}

\newcommand{\unit}{\operatorname{unit}}
\newcommand{\Ker}{\operatorname{Ker}}
\renewcommand\O{{\mathcal O}}
\newcommand\lb{\llbracket}
\newcommand\rb{\rrbracket}
\renewcommand{\phi}{\varphi}
\renewcommand{\leq}{\leqslant}
\renewcommand{\geq}{\geqslant}

\newcommand\ub{{\mathbf{u}}}

\newcommand{\lgw}{\longrightarrow}
\newcommand{\lgm}{\longmapsto}
\newcommand{\s}{\sigma}

\newcommand{\tg}{\sigma}
\newcommand{\ovl}{\overline}
\newcommand{\wdh}{\widehat}

\newcommand{\PP}{\mathbb P}
\newcommand{\ini}{\operatorname{in}}
\newcommand{\wdt}{\widetilde}
\newcommand{\het}{\operatorname{ht}}

\newcommand{\la}{\lambda}
\renewcommand{\O}{\mathcal{O}}
\renewcommand{\k}{\Bbbk}
\renewcommand{\a}{\alpha}
\renewcommand{\b}{\beta}

\newcommand{\Disc}{\operatorname{Disc}}

\newcommand{\K}{\mathbb{K}}
\newcommand{\g}{\gamma}

\newcommand{\ord}{\operatorname{ord}}
\renewcommand{\d}{\delta}

\newcommand{\Frac}{\operatorname{Frac}}

\newcommand{\KK}{\mathcal K}

\newcommand{\vb}{\mathbf{v}}
\newcommand{\x}{\mathbf{x}}
\newcommand{\y}{\mathbf{y}}
\newcommand{\OOx}{\KK\nl \x\nr}
\newcommand{\OOu}{\KK\nl \ub\nr}
\newcommand{\OOv}{\KK\nl \vb\nr}

\newcommand{\POx}{\mathbb P_h\nl\x\nr}

\newcommand{\w}{\mathbf{w}}
\newcommand{\om}{\omega}

\newcommand{\G}{\Gamma}

\newcommand{\nl}{\{\!\{}
\newcommand{\nr}{\}\!\}}

\newcommand{\Gr}{\operatorname{r}}
\newcommand{\Fr}{\operatorname{r}^{\mathcal{F}}}
\newcommand{\Ar}{\operatorname{r}^{\mathcal{A}}} 
\newcommand{\Tr}{\operatorname{r}^{\mathcal{W}}} 
\newcommand{\Wr}{\operatorname{r}^{\mathcal W}}

\newtheorem{theorem}{Theorem}[section]
\newtheorem{proposition}[theorem]{Proposition}
\newtheorem*{theorem*}{Theorem}

\newtheorem{corollary}[theorem]{Corollary}
\newtheorem{lemma}[theorem]{Lemma}

\theoremstyle{definition}

\newtheorem*{amalgamation*}{Amalgamation}

\newtheorem{remark}[theorem]{Remark}
\newtheorem{problem*}[theorem]{Problem}
\newtheorem*{remark*}{Remark}
\newtheorem{definition}[theorem]{Definition}

\usepackage{epigraph}

\begin{document}
\title{On rank Theorems for morphisms of local rings}
\author[A.~Belotto da Silva]{Andr\'e Belotto da Silva}
\author[O.~Curmi]{Octave Curmi}
\author[G.~Rond]{Guillaume Rond}
\address[A.~Belotto da Silva]{Université Paris Cité and Sorbonne Université, CNRS, IMJ-PRG, F-75013 Paris, France.}
\address[G.~Rond]{Universit\'e Aix-Marseille, Institut de Math\'ematiques de Marseille (UMR CNRS 7373), Centre de Math\'ematiques et Informatique, 39 rue F. Joliot Curie, 13013 Marseille, France.}
\address[O.~Curmi]{Alfréd Rényi Institute of Mathematics, Budapest, Hungary.}
\email[A.~Belotto da Silva]{andre.belotto@imj-prg.fr}
\email[O.~Curmi]{curmi@renyi.hu}
\email[G.~Rond]{guillaume.rond@univ-amu.fr}
\thanks{}



\begin{abstract}
We prove a generalization of Gabrielov's rank theorem for families of rings of power series which we call W-temperate. Examples include the families of complex analytic functions and of Eisenstein series. As a Corollary, we provide rank Theorems for convergent series in general characteristic zero complete valued fields (not necessarily algebraically closed, nor archimedean).
\end{abstract}

\maketitle

\section{Introduction}\label{sec:Intro}

A classical idea in local analytic geometry, which goes back at least to the Newton-Puiseux Theorem, consists of first studying a problem or an object formally, and then proving that the formal solution converges. Rank Theorems stand between the fundamental results of this philosophy applied to maps and equations, and they can be seen as the dual to the celebrated Artin approximation Theorem \cite{Artin,Artin2,DL}, cf. Remark \ref{rk:TemperateDualArtin}. We prove a rank Theorem for general families of rings which we call \emph{W-temperate}, see Theorem \ref{thm:TemperateRank}, generalizing the classical Gabrielov's rank Theorem \cite{Ga2,To2, BCR}. These are families of Weierstrass rings $(\KK\nl x_1,\ldots, x_r\nr)_{r\in \N}$, that is, families of rings of power series satisfying the Weierstrass division theorem, see Definition \ref{W-fam}, where $\KK$ is any uncountable algebraically closed field of characteristic zero, which satisfies three axioms: closure under local blowing-down, closure under restriction to generically hyperplane sections and temperateness, a closure under evaluation by algebraic elements type condition, see Definition \ref{temperate_fam}. 

Examples of Weierstrass families include the family of germs of complex-analytic functions, algebraic power series, Eisenstein power series, and convergent series over a complete non-archimedean field. Rank Theorems were known to hold in the first two examples \cite{Ga2,To2,BCR}, but our current methods greatly simplifies and shortens the proof in these contexts. Eisenstein power series have been systematically employed in the study of families of singularities, see e.g. \cite{ZarEisenstein,HirEqui,PP21}, going back at least to works of Zariski \cite[pp. 502]{ZarEisenstein}. In particular, they allow us to obtain rank Theorems for families of morphisms, see \cite[Theorem 3.5]{BCR2}. In particular, we use them in our follow-up work \cite{BCR2} to provide new proofs of two fundamental results of analytic and subanalytic geometry due to Paw\l ucki \cite{Pawthesis, Paw}, whose original proofs are considered very difficult \cite[pp. 1591]{LojCite}. Finally, the last example allows us to deduce rank Theorems for convergent power series in non-archimedean fields such as $\Q_p$, $\C_p$ and $\k(\!(t)\!)$, see Corollary \ref{thm:Gabrielov_convergent}. It can be seen as a new extension of a complex-analytic technique to the $p$-adic setting, c.f. \cite{DVdD}.

\bigskip
 
Let $\KK$ be a characteristic zero algebraically closed field. We consider families of rings of power series $(\KK\nl x_1,\ldots, x_r\nr)_{r\in \N}$  called Weierstrass, see Definitions \ref{W-fam}, which were introduced by Denef and Lipshitz \cite{DL}. Note that the completion of $\KK \nl x_1,\ldots, x_r \nr$ is $\KK \lb x_1,\ldots, x_r \rb $, see Proposition  \ref{rk:Wtemp}\,\ref{prop:hensel}. Given a ring homomorphism:
\[
\phi : \KK\nl \x \nr\lgw \KK\nl \ub \nr
\]
where $\x = (x_1,\ldots,x_n)$ and $\ub = (u_1,\ldots, u_m)$, we say that $\phi$ is a morphism of Weierstrass power series if $\phi(f) = f(\phi(\x))$ for every $f\in \KK\nl \x \nr$, see Definition \ref{def:TemperateMorphism}. We denote by $\wdh{\phi}$ its extension to the ring of formal power series. We define:
\begin{equation}\label{eq:DefRanks}
\begin{aligned}
\text{the Generic rank:} & &\Gr(\phi) &:= \mbox{rank}_{\mbox{Frac}(\KK \nl \ub \nr )}(\mbox{Jac}(\phi)),\\
\text{the Formal rank:}& &\Fr(\phi) &:= \dim\left(\frac{\KK \lb \x \rb  }{\Ker(\wdh{\phi})}\right),\\
\text{and the Weierstrass rank:}& &\Tr(\phi)&:=\dim\left(\frac{\KK \nl \x \nr}{\Ker(\phi)}\right),
\end{aligned}
\end{equation}
of $\phi$, where $\mbox{Jac}(\phi)$ stands for the matrix $[\partial_{u_i} \phi(x_j)]_{i,j}$. Our main result concerns Weierstrass families that satisfy three extra axioms which we call \emph{W-temperate}, see Definitions \ref{W-fam}:

\begin{theorem}[W-temperate rank Theorem]\label{thm:TemperateRank}
Let $\phi:\KK\nl \x \nr\lgw \KK\nl  \ub \nr$ be a morphism of rings of W-temperate power series. Then
\[
\Gr(\phi)=\Fr(\phi)\Longrightarrow \Gr(\phi)=\Fr(\phi)=\Tr(\phi).
\]
\end{theorem}

This result is proved in $\S$\ref{sec:Proof}. It generalizes the original rank Theorem of Gabrielov \cite{Ga2}, which concerns the case where $\KK\nl \x \nr$ stands for the family of complex analytic function germs. We rely on \cite{BCR} for a presentation of the importance and consequences of the Theorem to local analytic geometry and commutative algebra. The original proof of Gabrielov's rank Theorem is considered to be very difficult, cf. \cite[pp. 1]{Iz}. We have recently provided an alternative proof in \cite{BCR} by developing geometric-formal techniques inspired by works of Gabrielov \cite{Ga2} and Tougeron \cite{To2}. One of the difficulties involved in the proof is the intricate interplay between algebraic geometry and complex analysis. Our new result greatly simplifies and shortens the proof by addressing this difficulty: the proof of Theorem \ref{thm:TemperateRank} follows from algebraic geometry methods; complex analysis is only used in order to show that complex analytic functions form a W-temperate family, see $\S\S$ \ref{ssec:ComplexAnalyticTemperate}. As a mater of fact, we systematically generalize the arguments introduced in \cite{BCR} to their most general context, which demand us to introduce new commutative algebra arguments. It seems likely that the discussion of rank Theorems for non W-temperate families will demand a complete different strategy. In order to motivate this discussion, we provide a family of local rings of interest to function theory and tame geometry (that is, families of quasianalytic Denjoy-Carleman functions and families of $C^{\infty}$-definable functions over an o-minimal and polynomially bounded structure) where the rank Theorem does not hold, see $\S$\ref{ssec:Quasianalytic}.

As a Corollary of Theorem \ref{thm:TemperateRank}, we prove rank Theorems for convergent series in arbitrary complete valued (not necessarily algebraically closed) fields of characteristic zero $(\mathcal{K}, |\cdot|)$. In fact, recall that a power series $f=\sum_{\a\in\N^n}f_\a\x^\a\in\KK\lb \x\rb$ is \emph{convergent} if there exist real positive numbers $A$, $B>0$ such that
\[
\forall\a\in\N^n,\ |f_\a|A^{|\a|}\leq B.
\]
We denote by $\KK\{x_1,\ldots, x_n\}$ the subring of formal power series which are convergent (we follow here the convention used in commutative algebra, c.f. \cite{AM}, instead of non-archimedean analytic geometry, where $\KK\{x_1,\ldots, x_n\}$ would stand for convergent series with radius $1$). For a morphism $\phi:\KK\{\x\}\lgw \KK\{\ub\}$ of convergent power series, we denote by $\Ar(\phi)$ the analytic rank that is equal to $\dim\left(\fract{\KK\{\x\}}{\Ker(\phi)}\right)$.

\begin{corollary}[Gabrielov's rank Theorem for ring of convergent power series]\label{thm:Gabrielov_convergent}
Let $\KK$ be a complete valued field of characteristic zero.
Let $\phi: \KK\{\x\}\lgw \KK\{\ub\}$ be a morphism of convergent power series. Then
\[
\Gr(\phi)=\Fr(\phi)\Longrightarrow \Fr(\phi)=\Ar(\phi).
\]
\end{corollary}

In particular, this theorem applies to fields as $\KK=\R$ or $\C$, as discussed above and previously proved in \cite{Ga2,BCR}, but also to non-archimedean fields such as $\Q_p$, $\C_p$ and $\k(\!(t)\!)$. We prove this result in $\S\S$\ref{ssec:ComplexAnalyticTemperate}.

\begin{remark}\label{rk:TemperateDualArtin}
Theorem \ref{thm:TemperateRank} can be seen as a "dual" of the Artin approximation Theorem. More precisely, let $\phi$ be such that $\Gr(\phi)=\Fr(\phi)$. Then, 
$$
\forall F(\x)\in\KK\lb\x\rb, \text{ such that } \ F(\phi(\x))=0,
$$
$$
\forall c\in\N, \exists F_c(\x)\in\KK\nl\x\nr, F_c(\phi(\x))=0 \text{ and } F(\x)-F_c(\x)\in (\x)^c.
$$
where $\phi(\x) = (\phi(x_1),\ldots,\phi(x_n))$. Indeed, the ideals $\Ker(\phi)$ and $\Ker(\wdh\phi)$ are prime ideals of $\KK\nl\x\nr$ and $\KK\lb\x\rb$ respectively, and the equality $\Fr(\phi)= \Tr(\phi)$ is equivalent to the equality of the heights of these two ideals. Since $\KK\nl\x\nr$ is Noetherian (see Proposition \ref{rk:Wtemp} \ref{prop:hensel}), the height of $\Ker(\phi)\KK\lb\x\rb$ equals the height of $\Ker(\phi)$. Now, by Artin Approximation Theorem, see Corollary \ref{cor:Artin2}, $\Ker(\phi)\KK\lb\x\rb$ is again a prime ideal, so the equality $\Fr(\phi)= \Tr(\phi)$ is equivalent to the equality
\(
\Ker(\phi)\KK\lb\x\rb=\Ker(\wdh\phi).
\)
It is well known that, since $\KK\nl\x\nr$ is Noetherian, $\Ker(\phi)\KK\lb\x\rb$ is the closure of $\Ker(\phi)$ in $\KK\lb\x\rb$ for the $(\x)$-adic topology, and we conclude easily.
\end{remark}

\bigskip

We would like to thank Charles Favre for sketching the proof of Lemma \ref{lem:max_principle}, and Antoine Ducros and Lorenzo Fantini for answering our questions about non-archimedean fields and Berkovich geometry.
We would also like to thank Edward Bierstone for useful discussions. 
This work was supported by the CNRS project IEA00496  PLES. The first author is supported by the project ``Plan d’investissements France 2030", IDEX UP ANR-18-IDEX-0001. The second author thanks the grant NKFIH KKP 126683.

\section{Weierstrass Temperate families}
\subsection{W-Temperate families} 



Let $\KK$ be a field of characteristic zero. For every $n\in \mathbb{N}$, we denote by $(x_1,\ldots,x_n)$ indeterminates; we will use the compact notation $\x = (x_1,\ldots,x_n)$ and $\x' = (x_1,\ldots,x_{n-1})$ whenever there is no risk of confusion on $n$. We start by recalling the notion of Weierstrass family introduced in  \cite{DL}:

\begin{definition}\label{W-fam}
A \emph{Weierstrass family (over $\KK$)}, or just a W-family, of rings is a family $(\KK\nl x_1,\ldots, x_n\nr)_{n\in\N}$ of $\KK$-algebras such that,
\begin{enumerate}[label=\roman*)]
\item\label{axiom:inclusions} For every $n$,
\[
\KK[\x]\subset \KK\nl \x \nr\subset \KK\lb \x \rb.
\]

\item\label{axiom:intersection_n+m} For every $n$ and $m$, denoting $\x=(x_1,\ldots, x_n)$ and $\bold{y}= (y_1,\ldots,y_m)$:
\[
\KK\nl \x ,\bold{y}\nr\cap\KK\lb \x\rb=\KK\nl \x \nr.
\]
\item\label{axiom:permutation} For any permutation $\s$ of $\{1,\ldots, n\}$, and any $f\in\KK\nl \x\nr$, 
\[
f(x_{\s(1)},\ldots, x_{\s(n)})\in\KK\nl \x \nr.
\]

\item\label{axiom:unit} If $f\in\KK\nl \x \nr$ with $f(0)\neq 0$, then $f$ is a unit in $\KK\nl \x \nr$.

\item\label{axiom:weierstrass} The family is closed by \emph{Weierstrass division}. More precisely, let $F\in\KK\nl \x \nr$ be such that $F(0,x_{n})=x_{n}^du(x_{n})$ where $u(0)\neq 0$. For every $G\in\KK\nl \x \nr$,
\[
G=FQ+R,
\]
where $Q\in\KK\nl \x \nr$ and $R\in\KK\nl \x'\nr[x_{n+1}]$, $\deg_{x_{n}}(R)<d$, are unique.
\end{enumerate}
\end{definition}


A W-family satisfies several extra well-known properties which we recall in $\S\S$\ref{ssec:WProperties} below; in what follows we use these properties. Let us now provide the definition of W-temperate family:

\begin{definition}\label{temperate_fam}
Let $\KK$ be an uncountable algebraically closed  field of characteristic zero. A \emph{Weierstrass temperate family (over $\KK$)}, or just a W-temperate family, of rings is a Weierstrass family $(\KK\nl x_1,\ldots, x_n\nr)_{n\in\N}$ over $\KK$ satisfying the following three properties:
\begin{enumerate}[label=\roman*)]
\item\label{axiom:monomialcomposition} Closure by local blowings-down: For every $f\in \KK\lb \x \rb$, $n>1$, we have
\[
f(\x',x_1x_n)\in \KK\nl \x \nr \Longrightarrow f(\x)\in \KK\nl \x \nr.
\]
\item\label{axiom:Abhyankar-Moh} Closure by generic hyperplane sections: Let $F\in\KK\lb \x \rb\setminus \KK\nl \x \nr$. Set
\[
W:=\left\{\la\in \KK\mid  F(\x',\la x_1)\in \KK\nl\x'\nr\right\}.
\]
Then the set $\KK\setminus W$ is uncountable.

\item\label{axiom:last} Temperateness: Let $\x = (x_1,x_2)$ and $\alpha \in \N^{\ast}$. Let $\G(t,z)\in\KK[t,z]$ be an irreducible polynomial that is monic in $z$, and assume that $\G$ splits in $\KK\nl t\nr[z]$. Let $\g(t)$ and $\g'(t)\in\KK\nl t\nr$ be two roots of $\G$ and consider
\[
 P(\x,z)= \sum_{k \in \mathbb{N}} x_1^{k} p_k(x_2,z)\in \KK\lb \x\rb[z],
\]
where $ p_k(x_2,z) \in \KK[x_2,z]$ is such that $\deg_{x_2}(p_k) \leq \alpha k$ for every $k\in \mathbb{N}$. Then
\[
P(\x,\g(x_2))\in\KK\nl \x \nr\Longrightarrow P(\x,\g'(x_2))\in\KK\nl \x\nr.
\]
\end{enumerate}
\end{definition}

Note that properties \ref{axiom:Abhyankar-Moh} and \ref{axiom:last} are used only once in the paper, see $\S\S$\ref{ssec:RedLowDim} and the proof of Theorem \ref{local_to_semiglobal} respectively. 


\subsection{Weierstrass morphisms and ranks} 

We start by providing a detailed definition of the morphisms we consider:

\begin{definition}\label{def:TemperateMorphism}
Let $(\KK\nl x_1,\ldots, x_n\nr)_n$ be a Weierstrass family (resp. a W-temperate familly) and let $\phi:\KK\nl \x \nr\lgw \KK\nl \ub \nr$ be a morphism of local rings. We call $\phi$ a \emph{morphism of rings of Weierstrass power series} (resp. \emph{of W-temperate power series}) if there exist  power series $\phi_1(\ub)$, \ldots, $\phi_n(\ub)\, \in (\ub)\KK\nl \ub\nr$ such that
\[
\forall f(\x)\in\KK\nl\x\nr,\ \ \phi(f)=f(\phi_1(\ub),\ldots, \phi_n(\ub)).
\]
\end{definition}

For such a morphism, we have introduced in the introduction three notions of ranks: generic, formal and Weierstrass (or temperate), see Equation \eqref{eq:DefRanks}. Note that the generic and formal ranks can be introduced, in an obvious way, for general morphisms of power series rings $\psi:\KK\lb \x\rb\lgw \KK\lb \ub\rb$. Let us start by showing that these ranks are well-defined:

\begin{lemma}\label{lem:Basic}
Let $\phi:\KK\nl \x \nr\lgw \KK\nl \ub \nr$ be a morphism of Weierstrass power series. Then $\Gr(\phi)$, $\Fr(\phi)$ and $\Wr(\phi)$ are natural numbers such that:
\[
\Gr(\phi)\leq \Fr(\phi) \leq \Wr(\phi).
\]
\end{lemma}
\begin{proof}
It is straightforward that $\Gr(\phi)$ is well-defined; $\Fr(\phi)$ and $\Wr(\phi)$ are well-defined since $\fract{\KK\lb \x\rb}{\Ker(\wdh{\phi})}$ and $\fract{\KK\nl\x\nr}{\Ker(\phi)}$ are Noetherian local rings. Next, consider a general morphism of power series ring $\psi:\KK\lb \x\rb\lgw \KK\lb \ub\rb$ and set $r=\Gr(\psi)$. Apart from re-ordering the coordinates, we may assume that the matrix $[\partial_{u_i} \phi(x_j)]_{i\leq m,j\leq r}$ has rank $r$. Therefore, if we set $R:=\KK\lb x_1,\ldots, x_r\rb$, $\psi_{|_R}$ is injective by \cite[Lemma 4.2]{Ga2} (whose proof remains valid over any characteristic zero field $\KK$). Thus $r=\dim(R)\leq \dim\left(\fract{\KK\lb \x\rb}{\Ker(\psi)}\right)$. This proves the first inequality. Finally, by Artin approximation Theorem \ref{Artin}, $\Ker(\phi)\KK\lb\x\rb$ is a prime ideal, and by \cite[Theorem 9.4]{Mat} the height of $\Ker(\phi)\KK\lb\x\rb$ is less than or equal to the height of $\Ker(\wdh\phi)$. We conclude that $\Fr(\phi)\leq\Wr(\phi)$.
\end{proof}

\begin{remark}
The proof of  Lemma \ref{lem:Basic} also shows the result for a general morphism of power series rings $\psi:\KK\lb \x\rb\lgw \KK\lb \ub\rb$, that is, its generic and formal ranks are well defined and:
\[
\Gr(\psi)\leq \Fr(\psi).
\]
\end{remark}

\begin{definition}\label{def:RegularWTmorphism}
Let $\phi:\KK\nl \x \nr\lgw \KK\nl \ub \nr$ be a morphism of Weierstrass power series. We say that $\phi$ is \emph{regular} (in the sense of Gabrielov) if $\Gr(\phi) = \Fr(\phi)$. 
\end{definition}

We finish this subsection by useful results about the ranks of a morphism of Weierstrass power series, which is a Weierstrass version of \cite[Prop. 2.2]{BCR}:

\begin{proposition}[{cf. \cite[Prop.2.2]{BCR}}]\label{prop:inv_ranks}
Let $\phi:\KK\nl \x \nr\lgw \KK\nl \ub \nr$ be a morphism of Weierstrass  power series. The ranks $\Gr(\phi)$, $\Fr(\phi)$ and $\Wr(\phi)$ are preserved if we compose $\phi$ with:
\begin{enumerate}
\item a morphism $\s:\KK\nl u_1,\ldots, u_m\nr \lgw \KK\nl u'_1,\ldots, u'_s\nr$ such that $\Gr(\s)=m$,
\item an injective finite morphism $\t:\KK\nl x'_1,\ldots, x'_n\nr\lgw \KK\nl x_1,\ldots,x_n\nr$,
\item an injective finite morphism $\t:\KK\nl x'_1,\ldots, x'_t\nr \lgw \fract{\KK\nl\x\nr}{\Ker(\phi)}$.
\end{enumerate}
\end{proposition}

\begin{proof}
We start by proving $(1)$. Note that it is straightforward from linear algebra that $\Gr(\sigma \circ \phi) = \Gr(\phi)$. In order to prove the other two equalities, it is enough to prove that $\sigma$ and $\wdh{\s}$ are injective morphisms of local rings. This follows from Lemma \ref{lem:Basic}, since $ m = \Gr(\s) \leq \Fr(\s) \leq\Wr(\sigma) \leq \dim\KK \nl u_1,\ldots, u_m \nr = m$.

We now prove that $\Wr(\phi) = \Wr(\phi \circ \tau)$ under the hypothesis given in (2) and (3). Indeed, we have $\Ker(\phi\circ\t)=\Ker(\phi)\cap \KK\nl\x'\nr$ because $\tau$ is injective. Since $\KK\nl\ub\nr$ is an integral domain, $\Ker(\phi)$ and $\Ker(\phi\circ\t)$ are prime ideals. Thus, by the Going-Down theorem for integral extensions \cite[Theorem 9.4ii]{Mat}, we have that $\het(\Ker(\phi\circ\t))\leq\het(\Ker(\phi))$, thus $\Wr(\phi)\leq\Wr(\phi\circ\t)$. On the other hand, we have the equality $\Wr(\phi)=\Wr(\phi\circ\t)$ because $\het(\Ker(\phi\circ\t))=\het(\Ker(\phi))$ by \cite[Theorem 9.3]{Mat}.

We now prove that $\Fr(\phi) = \Fr(\phi \circ \tau)$ under the hypothesis given in (2). Indeed, since $\t$ is finite, $\wdh\t$ is also finite by the Weierstrass division Theorem (see for instance \cite[Cor. 1.10]{BCR} for this claim). Moreover, we have 
\[
\dim(\KK\lb \x'\rb)-\het(\Ker\wdh{\t})=\dim(\KK\lb\x\rb) =n
\]
since finite morphisms preserve the dimension and $\t$ is injective. But $\het(\Ker(\wdh{\t}))=0$ if and only if $\Ker(\wdh{\t})=(0)$ because $\KK\nl\x'\nr$ is an integral domain. Thus, $\wdh{ \t}$ is injective and $\Fr(\phi\circ\t)=\Fr(\phi)$.

Now we prove that $\Fr(\phi) = \Fr(\phi \circ \tau)$ under the hypothesis given in (3). We denote by $\wdh \t'$ the morphism induced by $\wdh\t$:
$$\fract{\KK\lb x'_1,\ldots, x'_t\rb}{\wdh\t^{-1}(\Ker(\wdh\phi))} \lgw \fract{\KK\lb\x\rb}{\Ker(\wdh\phi)}.$$
As in the previous case, since $\t$ is finite, $\wdh\t$ is also finite, therefore $\wdh\t'$ is finite. Moreover, by definition, $\wdh\t'$ is injective. Thus, by  \cite[Theorem 9.3]{Mat}, we have $\Fr(\phi\circ\t)=\Fr(\phi)$.

We now turn to the proof of $\Gr(\phi) = \Gr(\phi \circ \tau)$. We start by (2). Let $J_\phi$ and $J_\t$ denote the Jacobian matrices of $\phi$ and $\t$. Then we have $J_{\phi\circ\t}=J_\phi  \cdot\phi(J_\t)$; note that it is enough to prove that the hypothesis imply that $\Gr(\t) = n$ in order to conclude by standard linear algebra. Indeed, since $\t$ is finite and injective, for every $i\in\{1,\ldots, n\}$, there is a monic polynomial $P_i(\x',x_i)\in\KK\nl\x'\nr[x_i]$ such that
$$
P_i(\t_1(\x'),\ldots, \t_n(\x'),x_i)=0 \text{ and }\dfrac{\partial P_i}{\partial x_i}(\t_1(\x'),\ldots, \t_n(\x'),x_i) \neq 0.
$$
Therefore, for every $i$ and $j$, we have
$$\sum_{\ell=1}^n\frac{\partial P_i}{\partial x'_\ell}(\t(\x),x_i)\frac{\partial \t_\ell(\x)}{\partial x_j}=\left\{\begin{array}{ccc}-\dfrac{\partial P_i}{\partial x_i} & \text{ if } &i=j\\ 0 & \text{ if }& i\neq j\end{array}\right.$$
Thus 
$$
\begin{bmatrix} \dfrac{\partial P_i}{\partial x'_1} & \cdots& \dfrac{\partial P_i}{\partial x'_n}\end{bmatrix}\cdot J_\t=-\dfrac{\partial P_i}{\partial x_i}\cdot e_i
$$
where $e_i$ is the vector whose coordinates are zero except the $i$-th one which is equal to 1. In particular $J_\t$ is generically a matrix of maximal rank, that is $\Gr(\tau) = n$, so $\Gr(\phi)=\Gr(\phi\circ\t)$ by standard linear algebra. This proves (2).

Finally, let us finish the proof of (3). By adding  the $x'_i$ to the $x_j$, we can assume that $x'_i=x_i$ for $i\leq t$. By assumption, for every $i\in\{1,\ldots, n\}$, there is a monic polynomial $P_i(\x',x_i)\in\KK\nl\x'\nr[x_i]$ such that
$$
P_i(\t_1(\x),\ldots, \t_n(\x),x_i)=f_i(\x)\in\Ker(\phi) \text{ and }\dfrac{\partial P_i}{\partial x_i}(\t_1(\x'),\ldots, \t_n(\x'),x_i) \notin \Ker(\phi).
$$
Thus 
$$
\begin{bmatrix} \dfrac{\partial P_i}{\partial x'_1} & \cdots& \dfrac{\partial P_i}{\partial x'_n}\end{bmatrix}\cdot J_\t=-\dfrac{\partial P_i}{\partial x_i}\cdot e_i+\begin{bmatrix} \dfrac{\partial f_i}{\partial x_1}(\x) & \cdots & \dfrac{\partial f_i}{\partial x_n}(\x)\end{bmatrix}
$$
Since $f_i\in\Ker(\phi)$, we have $f_i(\phi(\ub))=0$. By differentiation we obtain
$$\forall j=1,\ldots, m,\ \ \sum_{k=1}^n\frac{\partial f_i}{\partial x_k}(\phi(\ub))\frac{\partial \phi_k(\ub)}{\partial u_j}=0,$$
that is, 
$$\begin{bmatrix} \dfrac{\partial f_i}{\partial x_1}(\phi(\ub)) \ \cdots\  \dfrac{\partial f_i}{\partial x_n}(\phi(\ub))\end{bmatrix}\cdot J_\phi=0.$$
This proves that the generic rank of $J_{\phi\circ\t}= J_\phi \cdot \phi(J_\t)$ is the rank  of $J_\phi$.
\end{proof}


\subsection{Properties of Weierstrass families}\label{ssec:WProperties}

We now recall several useful properties of W-families which are either proved in \cite{DL,Ro1} (see precise references in the proof), or which follow easily from classical results:

\begin{proposition} \label{rk:Wtemp}
 Let $(\KK\nl x_1,\ldots, x_n\nr)_{n\in\N}$ be a Weierstrass family. Then the following properties are satisfied:
\begin{enumerate}[label=\roman*)]
\item \label{prop:hensel} For every $n$, $\KK\nl \x \nr$ is a Henselian, Noetherien, UFD  regular local ring whose maximal ideal is generated by $(x_1,\ldots, x_n)$, and completion is $\KK\lb \x \rb$.
\item\label{prop:composition} For  $f\in \KK\nl t, \x \nr$ and any $g\in (\x) \, \KK\nl \x \nr$,
\(
f(g,\x)\in \KK\nl \x \nr.
\)
\item\label{rk:ramif} For every $f\in \KK\lb \x \rb$, and any $q\in\N^*$, we have
\[
f(\x', x_n^q)\in\KK\nl \x \nr\Longrightarrow f(\x)\in \KK\nl \x\nr.
\]
\item\label{rk:division} For every $n$ and $k\leq n$,
\[
\KK\nl \x \nr\cap (x_k)\, \KK\lb \x \rb=(x_k)\, \KK\nl \x \nr.
\]
\item\label{Wprep} Weierstrass preparation Theorem: Let $f\in\KK\nl \x \nr$ be such that $f(0,\ldots, 0,x_n)\neq 0$ has order $d$ in $x_n$. Then there exists a unit $U$ and a Weierstrass polynomial $P=x_n^d+a_1(\x')x_n^{d-1}+\cdots+a_d(\x')$ such that
\[
f(\x)=U(\x)\cdot (x_n^d+a_1(\x')x_n^{d-1}+\cdots+a_d(\x')).
\]
\item\label{rk:noethernorm} Noether Normalization: Let $A=\fract{\KK\nl \x \nr}{I}$ where $I$ is an ideal of $\KK\nl \x \nr$. Then, apart from a linear change of indeterminates $x_1$, \ldots, $x_n$, there exists an integer $r>0$ such that the canonical morphism 
\[
\KK\nl x_1,\ldots, x_r\nr\lgw \fract{\KK\nl x_1,\ldots, x_n\nr}{I}
\]
is finite. Moreover, since the dimension does not change under finite morphisms, if $\dim\left(\fract{\KK\nl \x\nr}{I}\right)=r$, then $\operatorname{ht}(I)=n-r$.
\end{enumerate}
\end{proposition}
\begin{proof}
Properties \ref{prop:hensel}, \ref{prop:composition}, \ref{rk:division} are given in \cite[Remark 1.3]{DL}; property \ref{rk:ramif} is given in \cite[Lemme 5.13]{Ro1}. To prove property \ref{Wprep}, it is enough to consider the Weierstrass division of $x_n^d$ by $f(\x)$. Finally, it is classical that property \ref{rk:noethernorm} follows from the Weierstrass division Theorem, see e.g. \cite[3.319]{JP}.
\end{proof}

We add to this list of properties the following two Theorems:

\begin{theorem}[{\cite[Theorem 5.5]{PR}}]\label{A-J} A Weierstrass  family of rings satisfies the Abhyankar-Jung Theorem. More precisely, let $P(\x,Z)\in \KK\nl \x\nr[Z]$ be a monic polynomial in $Z$. Assume that
\[
\Disc_Z(P)=\x^\a u(\x)
\]
where $u(\x)\in \KK\nl \x\nr$ satisfies $u(0)\neq 0$. Then there is  $q\in\N^*$ such that the roots of $P$ belong to $\KK\nl x_1^{1/q},\ldots, x_n^{1/q}\nr$.
\end{theorem}

The next result motivates the introduction of the notion of W-families in \cite{DL}.

\begin{theorem}[{\cite[Theorem 1.1]{DL}}]\label{Artin}
 A Weierstrass  family of rings satisfies the Artin Approximation Theorem: let $F=(F_1, \ldots, F_p)\in\KK\nl \x\nr[\y]^p$ with $\y=(y_1,\ldots, y_m)$, and let $\wdh g(\x)=(\wdh g_1(\x), \ldots, \wdh g_m(\x))\in\KK\lb\x\rb^m$ be a formal power series solution:
$$F(\x,\wdh g(\x))=0.$$
Let $c\in\N$. Then there is a solution $g^{(c)}(\x)=(g_1^{(c)}(\x),\ldots, g_m^{(c)}(\x))\in\KK\nl\x\nr^m$:
$$F(\x, g^{(c)}(\x))=0$$
with $g_i^{(c)}(\x)-\wdh g_i(\x)\in (\x)^c$ for every $i$.
\end{theorem}

In what follows, we will use the following well known corollaries of the above result (and we provide their proofs for the sake of completeness).

\begin{corollary}\label{cor:Artin2}
Let $\phi: \KK \nl \x \nr \lgw\KK \nl \ub \nr$ be a morphism of Weierstrass power series. Then $\Ker(\phi)\KK\lb\x\rb$ is a prime ideal.
\end{corollary}
\begin{proof}
The following is a well-known argument. Let $\wdh f$, $\wdh g\in\KK\lb \x\rb$ be such that $\wdh f\wdh g\in  \Ker(\phi)\KK\lb\x\rb$. That is, there exist $f_1$, \ldots, $f_s\in\Ker(\phi)$ and $\wdh h_1$, \ldots, $\wdh h_s\in\KK\lb \x\rb$ such that
$$\wdh f\wdh g-\sum_{i=1}^s f_i\wdh h_i=0.$$
By Artin approximation Theorem applied to $y_{s+1}y_{s+2}-\sum_{i=1}^sf_i y_i$, for every $c\in\N^*$, there exist $f^{(c)}$, $g^{(c)}$, $h^{(c)}_1$, \ldots, $h^{(c)}_s\in\KK\nl \x\nr$ such that
$$f^{(c)}g^{(c)}-\sum_{i=1}^s f_i h_i^{(c)}=0$$
and  $\wdh f-f^{(c)}$, $\wdh g- g^{(c)}\in (\x)^c$. Since $\Ker(\phi)$ is a prime ideal, then $f^{(c)}$ or $g^{(c)}$ is in $\Ker(\phi)$. Apart from replacing $f$ by $g$, we may assume that $f^{(c)}\in \Ker(\phi)$ for infinitely many $c$. Therefore, $f$ is the limit of elements of $\Ker(\phi)$, that is, $f$ belongs to the closure of $\Ker(\phi)$ in $\KK\lb\x\rb$ for the $(\x)$-topology. But, by \cite[Theorem 8.11]{Mat},  this closure is exactly $\Ker(\phi)\KK\lb \x\rb$, so $f\in\Ker(\phi)\KK\lb \x\rb$. This proves that $\Ker(\phi)\KK\lb\x\rb$ is a prime ideal.
\end{proof}

\begin{corollary}\label{cor:Artin}
Let $(\KK\nl x_1,\ldots, xn\nr)_n$ be a Weierstrass family.
Suppose that $P$ and $Q$ are monic polynomials in $y$, $P\in\KK\lb\x\rb[y]$ and $Q\in\KK\nl\x\nr[y]$, that are not coprime in $\KK\lb\x,y\rb$. Then $P$ and $Q$ admits a common W-factor $R\in\KK\nl\x\nr[y]$.
\end{corollary}
\begin{proof}
By hypothesis, there is a non unit $R\in\KK\lb\x,y\rb$ that divides $P$ and $Q$ in $\KK\lb\x,y\rb$. Since $P$ is monic in $y$, then $R(0,y)\neq 0$, so $R$ equals a unit times a monic polynomial by Weierstrass preparation for formal power series. By replacing $R$ by this monic polynomial we may assume that $R$ is monic in $y$. So we have $RS=Q$ where $S\in\KK\lb\x\rb[y]$ is monic in $y$. We write 
\[
R=\sum_{i=0}^dr_i(\x)y^i, \quad S=\sum_{i=0}^es_i(\x)y^i \quad \text{ and } Q=\sum_{i=0}^{d+e}q_i(\x)y^i.
\]
The equality $RS=Q$ is equivalent to the system of equations
\[
\sum_{k=\max\{0,\ell-e\}}^{\min\{\ell,d\}} r_k(\x)s_{\ell-k}(\x)-q_\ell(\x)=0 \text{ for } \ell=0,\ldots, d+e. 
\]
By Artin approximation Theorem \ref{Artin}, for any $c\in\N$, this system of equations has a solution $(r'_i(\x),s'_j(\x))\in\KK\nl\x\nr^{d+e+2}$ and that coincide with $(r_i(\x),s_i(\x))$ up to $(\x)^c$. We set $R'(\x,y)=\sum_{i=0}^d r_i'(\x)y^i$. Since $\KK\nl\x\nr$ is a UFD by Proposition \ref{rk:Wtemp}\,\ref{prop:hensel}, $Q$ has finitely many  monic factors of degree $d$ in $y$ that we denote by $R_1$, \ldots, $R_s$. Let us choose $c\in\N$ large enough to insure that $R_i-R_j\notin (\x)^{c}$ when $i\neq j$. Since $R'$ equals one of the $R_i$,   necessarily $R'=R$. This proves that $R\in\KK\nl\x\nr[y]$, so $P$ and $Q$ admits a common W-factor $R\in\KK\nl\x\nr[y]$.
\end{proof}


\section{Proof of W-temperate rank Theorem}\label{sec:Proof}
\subsection{Reduction of Theorem \ref{thm:TemperateRank} to Theorem \ref{thm:red1}}\label{red:sec}
We start by proving by contradiction that the following result implies Theorem \ref{thm:TemperateRank}:

\begin{theorem}\label{thm:red1}
Let $\phi:\KK\nl \x, y\nr\lgw \KK\nl \ub \nr$ be a morphism of rings of W-temperate power series  such that
\begin{enumerate}
\item[i)] The kernel of $\phi$ is generated by one Weierstrass polynomial $P\in\KK\lb \x\rb[y]$.
\item[ii)] $\Gr(\phi)=\Fr(\phi)=n$.
\end{enumerate}
Then $P\in\OOx[y]$.
\end{theorem}

\begin{remark}\label{rk:red1}
In fact we prove here that the statement of Theorem \ref{thm:red1} implies the statement of Theorem \ref{thm:TemperateRank} when $(\KK\nl x_1,\ldots, x_n\nr)_n)$ is a Weierstrass family (and not only a W-temperate family). Indeed, we will need this reduction in the case of the family of convergent power series over a (not necessarily algebraically closed) complete field of characteristic zero, which is a Weierstrass family but not a W-temperate family.
\end{remark}

\begin{proof}[Reduction of Theorem \ref{thm:TemperateRank} to Theorem \ref{thm:red1}] 
We follow closely \cite[page 1347]{BCR}. Assume that Theorem \ref{thm:TemperateRank} does not hold, that is, there exists a morphism of rings of Weierstrass power series $\phi: \KK\nl\x \nr \lgw \KK\nl\ub\nr$, where $\x = (x_1,\ldots,x_n)$ and $\ub=(u_1,  \ldots,u_m)$, such that $\Gr(\phi)=\Fr(\phi) \geq 1$, but $\Fr(\phi)<\Wr(\phi)$. Consider the induced injective morphism $\fract{\KK\nl\x \nr}{\Ker(\phi)} \lgw \KK\nl\ub\nr$ and, by the Noether normalization given in Proposition \ref{rk:Wtemp} \ref{rk:noethernorm}, there exists a finite injective morphism $\t:\KK\ \nl \tilde{\x} \nr\lgw \fract{\KK\nl\x \nr}{\Ker(\phi)}$. By Proposition \ref{prop:inv_ranks}, we can replace $\phi$ by $\phi\circ\tau$, that is, we may assume that $\phi$ is injective.

Next, since $\Fr(\phi)<\Wr(\phi)=m$, we know that $\Ker(\wdh{\phi})\neq (0)$. Now, suppose that $\Ker(\wdh{\phi})$ is not principal or, equivalently, that its height is at least $2$. By the normalization theorem for formal power series, after a linear change of coordinates, the canonical morphism 
\[
\pi:\KK\lb x_1,  \ldots, x_{\Gr(\phi)}\rb\lgw \frac{\KK\lb \x\rb}{\Ker(\wdh\phi)}
\] 
is finite and injective. Thus,  the ideal $\mathfrak p:=\Ker(\wdh\phi)\cap\KK\lb x_1,  \ldots, x_{\Gr(\phi)+1}\rb$ is a nonzero height one prime ideal. Because $\KK\lb x_1,  \ldots, x_{\Gr(\phi)+1}\rb$ is a unique factorization domain,  $\mathfrak p$ is a principal ideal (see \cite[Theorem 20.1]{Mat} for example). After a linear change of coordinates, we may assume that $\mathfrak p$ is generated by a Weierstrass polynomial $P\in\KK\lb x_1,\ldots, x_{\Gr(\phi)}\rb[x_{\Gr(\phi)+1}]$.

Now, denote by $\phi'$ the restriction of $\phi$ to $\KK\nl x_1,  \ldots, x_{\Gr(\phi)+1}\nr$. By definition $P$ is a generator of $\Ker(\wdh\phi')$, thus $\Fr(\phi')=\Gr(\phi)+1-1=\Gr(\phi)=\Fr(\phi)$. Since $\phi$ is injective, $\phi'$ is injective and $P$ does not belong to $\KK\nl x_1,\ldots, x_{\Gr(\phi)}\nr[x_{\Gr(\phi)+1}]$. Moreover, since $\pi$ is finite, we can use again  Proposition \ref{prop:inv_ranks}, to see that
$$
\Gr(\phi')=\Gr(\wdh{\phi'})=\Gr(\wdh{\phi})=\Gr(\phi).
$$
 Therefore we have $\Gr(\phi')=\Fr(\phi')=m-1$, contradicting Theorem \ref{thm:red1}. 
 \end{proof}
\subsection{Reduction to the low-dimensional case}\label{ssec:RedLowDim}
We now prove by contradiction that the following result implies Theorem \ref{thm:red1}:
\begin{theorem}\label{thm:red2}
Let $\phi:\KK\nl x_1, x_2, y\nr\lgw \KK\nl u_1,u_2\nr$ be a morphism of rings of W-temperate power series such that
\begin{enumerate}
\item[i)] $\phi(x_1)=u_1$ and $\phi(x_2)=u_1u_2$,
\item[ii)]
$\Ker(\wdh\phi)$ is generated by one Weierstrass polynomial $P\in\KK\lb \x\rb[y]$.
\end{enumerate}
Then  $P\in\OOx[y]$.
\end{theorem}

\begin{proof}[Reduction of Theorem \ref{thm:red1} to Theorem \ref{thm:red2}]
We follow closely \cite[3rd Reduction]{BCR}. Assume that there is a morphism $\phi$ satisfying the hypothesis of Theorem \ref{thm:red1} but where $P\notin \OOx[y]$. 

\medskip
\noindent
$(1)$ First, after a linear change of coordinates in $\ub$ we may assume that \linebreak$\phi(x_1) (u_1,0,\ldots, 0)\neq 0$. Thus, the morphism  $\s\circ\phi$, where $\s$ is given by
$$
\s(u_1)=u_1\text{ and } \s(u_i)=u_1u_i \ \ \forall i>1
$$
 satisfies the hypotheses of Theorem \ref{thm:red1} and its kernel is generated by $P$, by Proposition \ref{prop:inv_ranks}. Thus we may assume that $\phi(x_1)=u_1^e\times U(\ub)$ where $U(\ub)$ is a unit. And by replacing $x_1$ by $\frac{1}{U(0)}x_1$ we may further assume that $U(0)=1$.
 
 \medskip
\noindent
 $(2)$ We define the morphism $\t$ by
 $$
 \t(x_1)=x_1^e \text{ and } \t(x_i)=x_i\ \ \forall i>1.
 $$
 Let $V(\ub)\in\KK\lb \ub\rb$ be a power series such that $V(\ub)^e=U(\ub)$. Such a power series exists since $U(0)=1$ and, by the Implicit Function Theorem cf. Proposition \ref{rk:Wtemp} \ref{prop:hensel}, $V(\ub)\in\OOu$. We define the morphism $\psi$ by
 $$
 \psi(x_1)=u_1V(\ub)\text{ and } \psi(x_i)=\phi(x_i)\ \ \forall i>1.
 $$
Then $\psi\circ\t=\phi$ and $P(x_1^e,x_2,\ldots, x_n,y)\in\Ker(\wdh\psi)$. Since $P(x_1^e,x_2,\ldots, x_n,y)$ is a Weierstrass polynomial in $y$, $\Ker(\wdh\psi)$ is generated by a Weierstrass polynomial $Q$ that divides $P$. Thus $P$ is the product of $Q$ with the distinct polynomials $Q(\xi x_1,x_2,\ldots, x_n,y)$ where $\xi$ runs over the $e$-th roots of unity. Therefore, if $Q\in\OOx[y]$,  $P\in\OOx[y]$ which contradicts the hypothesis. Thus, $\psi$ satisfies the hypothesis of Theorem \ref{thm:red1} but $\Ker(\wdh\psi)$ is generated by a Weierstrass polynomial that is not in $\OOx[y]$. By Proposition \ref{prop:inv_ranks}, we may replace $\phi$ by $\psi$ and assume that $\psi(x_1)=x_1$ by composing $\psi$ by the inverse of the temperate automorphism that sends $u_1$ onto $u_1V(\ub)$.

\medskip
\noindent
$(3)$ Now we have $\phi(x_1)=u_1$ and we perform ``Gabrielov's trick", cf. \cite[Example 3.5]{BCR}. We denote by $\phi_i(\ub)$ the image of $x_i$ by $\phi$. We consider the temperate automorphism $\chi$ defined by
$$\chi(x_1)=x_1\text{ and } \chi(x_i)=x_i-\phi_i(x_1,0,\ldots, 0)\ \ \forall i>1.$$
If we replace $\phi$ by $\phi\circ\chi$ we may assume that every nonzero monomial of $\phi(x_i)$ is divisible by one of the $u_i$ for $i>1$. Then by replacing $\phi$ by $\s\circ\phi$, where $\s$ is defined above, we may assume that $\phi(x_i)=u_1^{a_i}g_i(\ub)$ where $a_i>0$, $g_i(0)=0$ and $g_i(0,x_2,\ldots, x_n)\neq 0$, for $i>1$. Moreover, by composing with the morphism 
\[
x_i\lgm x_i^{\prod_{k\neq i}a_k}
\]
for $i>1$, we may assume that $a_i=a$ is independent of $i$. Finally, by replacing $x_1$ by $x_1^{a+1}$ we may assume that $\phi(x_1)=u_1^{a+1}$. Composing $\phi$ with these morphisms does not change the ranks, by Proposition \ref{prop:inv_ranks}.

\medskip
\noindent
$(4)$ Now we set, for $\la=(\la_2,\ldots, \la_n)\in\KK^{n-1}\setminus\{0\}$, $h_\la=x_1-\sum_{i=2}^n \la_ix_i$. We have
$$\phi(h_\la)=u_1^ag_\la(\ub)$$
where  $g_\la(\ub)=u_1-\sum_{i=2}^n\la_ig_i(\ub)\in\OOu$. By the implicit function Theorem, there exists a unique nonzero $\xi_\la(u_2,\ldots, u_n)\in\OOu$ such that $\xi_\la(0)=0$ and 
$$g_\la(\xi_\la(x_2,\ldots, x_n),x_2,\ldots, x_n)=0.$$
Let $M(\ub)$ be a nonzero minor of the Jacobian matrix of $\phi$ that is of maximal rank. Then assume that $M(\ub)$ is divisible by $h_\la(\ub)$ for every $\la\in\Lambda$, where $\Lambda$ cannot be written as a finite union of sets included in proper affine subsets of $\KK^{n-1}$. Thus, because $\KK\lb\ub\rb$ is a UFD, there is  a finite number of subsets $\Lambda_k\subset\KK^{n-1}$, $k=1$, \ldots, $N$, whose union equals $\Lambda$, and such that for every $\la$, $\la'\in \Lambda_k$, $h_\la$ and $h_{\la'}$ are equal up to multiplication by a unit. Thus, by the assumption on $\Lambda$, there is $k$ such that 
$\Lambda_k$ contains $n$ vectors of $\KK^{n-1}$, denoted by $\la^{(1)}$, \ldots, $\la^{(n)}$ such that the vectors $\la^{(i)}-\la^{(1)}$ are $\KK$-linearly independent. Therefore, there are units $U_i(\ub)$, for $i=2$, \ldots, $n$, such that
$$u_1-(g)\cdot\la^{(i)}=U_i(\ub)(u_1-(g)\cdot\la^{(1)})\ \ \forall i=2,\ldots, n$$
where $(g)$ denote the vector whose entries are the $g_i(\ub)$. This implies that the $(g)\cdot (\la^{(i)}-\la^{(1)})$ are divisible by $h_{\la^{(1)}}$. But the $\la^{(i)}-\la^{(1)}$ being $\KK$-linearly independent, every $g_i(\la)$ is divisible by $h_{\la^{(1)}}$, thus $u_1$ is divisible by $h_{\la^{(1)}}$, which implies that $\xi_{\la^{(1)}}=0$ contradicting the assumption. Therefore, there is a finite union of proper affine subspaces of $\KK^{n-1}$, denoted by  $\Lambda$, such that   $M(\ub)$ is not divisible by $h_\la$ for every $\la\in\KK^{n-1}\setminus\Lambda$. In particular $\KK^{n-1}\setminus\Lambda$ is uncountable.

After a linear change of coordinates, we may assume that $\KK\times\{0\}^{n-2}$ is not included in $\Lambda$, in particular $(\KK\times\{0\}^{n-2})\cap\Lambda$ is finite.
For any $(\la,0,\ldots, 0)\in(\KK\times\{0\}^{n-2})\setminus\Lambda$ the morphism
$$
\psi_\la:\frac{\KK\lb x_1,\ldots , x_n\rb[y]}{(x_1-\la x_2)}\lgw \frac{\KK\lb u_1,\ldots, u_n\rb[y]}{(u_1-\la g_2(\ub))}
$$
 is of rank $\Gr(\phi_\la)=n-1$. Then if $n>3$, by Bertini's Theorem \cite[Theorem 3.4]{BCR}, and by Definition \ref{W-fam} \ref{axiom:Abhyankar-Moh} (note that this is the only point of the paper where we use Definition \ref{W-fam} \ref{axiom:Abhyankar-Moh}), the polynomial $P$ remains irreducible and not in $\OOx[y]$ in $$\frac{\KK\lb x_1,\ldots , x_n\rb[y]}{(x_1-\la x_2)}\simeq\KK\lb x_2,\ldots, x_n\rb[y]$$ when $\la$ belongs to $W\subset \KK$ that is uncountable. Therefore we can choose $(\la,0,\ldots, 0)\in (W\times\{0\}^{n-2})\setminus\Lambda$, and this allows us replace $n$ by $n-1$ in Theorem \ref{thm:red1}. By repeating this process, we construct an example of a morphism $\phi$ with $n=2$ satisfying  Theorem \ref{thm:red2} (i) such that $\Ker(\wdh\phi)$ is generated by a Weierstrass polynomial that is not $\OOx[y]$; note that we must stop the reduction at $n=2$, because Bertini's Theorem does not hold for $n<3$, cf. \cite[Remark 3.6(3)]{BCR}. Moreover, by repeating the argument given in part (2) if necessary, we may assume that $\phi(x_1)=u_1$, and $\phi(x_2)$ has the form $u_1^ag(\ub)$ with $g(0)=0$ and $g(0,u_2)\neq 0$.
 
 By composing $\phi$ with the morphism $\s$ defined in (1), we can assume that $g(\ub)=u_2^bU(\ub)$ for some unit $U(\ub)$. Now let $\s'$ be the morphism  defined by $\s'(u_1)=u_1^b$ and $\s'(u_2)=u_1u_2^{a+1}$. Then, we have
$$\s'\circ\phi(x_1)=u_1^b\text{ and } \s'\circ\phi(x_2)=(u_1u_2)^{b(a+1)}V( \ub)$$
for some unit $V( \ub)$. Therefore, as done in (2), we can assume that
$$
\phi(x_1)=u_1\text{ and } \phi(x_2)=u_1u_2.
$$
Hence, we have constructed a morphism $\phi$ that satisfies the hypothesis of Theorem \ref{thm:red2}, but $\Ker(\wdh\phi)$ is generated by Weierstrass polynomial in $y$ that is not in $\KK\nl \x\nr[y]$, contradicting Theorem \ref{thm:red2}.
\end{proof}

The rest of this section is devoted to the proof of Theorem \ref{thm:red2}, given in $\S\S$\ref{ssec:red2}.

\subsection{Newton-Puiseux-Eisenstein Theorem}
In \cite[Section 5]{BCR}, we provided a framework allowing us to obtain a good factorization of a polynomial in $\C\lb\x\rb[y]$. We recall here the main definitions and adapt the main results to the more general context of polynomials in $\KK\lb \x\rb[y]$.

Consider the ring of power series $\KK\lb \x \rb$ where $\x=(x_1,  \ldots, x_n)$ and denote by $\KK(\!(\x)\!)$ its field of fractions. We denote by $\nu$ the $(\x)$-adic valuation on $\KK\lb \x\rb $. The valuation $\nu$ extends to $\KK(\!(\x)\!)$ by defining $\nu(f/g)=\nu(f)-\nu(g)$ for every $f$, $g\in\KK\lb \x\rb$, $g\neq 0$. Denote $V_\nu$ the valuation ring of $\nu$ in $\KK(\!(\x)\!)$, and $\wdh V_\nu$ its completion with respect to $\nu$. Let us now recall the notion of homogeneous element:

\begin{definition}[Homogeneous elements]\label{homogeneous_element}
Let $\omega=p/e\in \Q_{>0}$. We say that $\G\in \KK[\x,z]$ is \emph{$\omega$-weighted homogeneous} if $\G(x_1^e,\cdots,x_n^e,z^p)$ is homogeneous.

A \emph{homogeneous element} $\g$ is an element of an algebraic closure of $\KK(\x)$, satisfying a relation of the form $\Gamma(\x,\g)=0$ for some $\omega$-weighted homogeneous polynomial $\Gamma(\x,z)$, where $\omega \in \Q_{> 0}$. Furthermore, if $\Gamma(\x,z)$ is monic in $z$, we say that $\g$ is an \emph{integral homogeneous element}. In this case, $\omega$ is called the \emph{degree} of $\g$. 
\end{definition}

Given an integral homogeneous element $\gamma$ of degree $\omega$, there exists an extension of the valuation $\nu$, still denoted by $\nu$, to the field $\KK(\x)[\g]$, defined by
$$
\nu\left(\sum_{k=0}^{d-1}a_{k}(\x)\g^{k}\right)=\min\{\nu(a_{k})+k\omega\}.
$$
where $d$ is the degree of the field extension $\KK(\x) \lgw \KK(\x)[\g]$. We denote $V_{\nu,\g}$ the valuation ring of $\nu$ in $\KK(\x)[\g]$, and $\wdh V_{\nu,\g}$ its completion with respect to $\nu$.

\begin{definition}[Projective rings and temperate projective rings]
\label{def:ProjRings}Let $h\in \KK[\x]$ be a homogeneous polynomial. Denote by $\PP_h(\!(x)\!)$ the ring  of elements $A$  for which there is $k_0 \in \Z$, $\a$, $\b\in\N$ and $a_k(\x)$ homogeneous polynomials in $\KK[x]$ for $k\geq k_0$ such that:
\[
A = \sum_{k\geq k_0} \frac{a_k(\x)}{h^{\a k+\b}}, \quad \text{ where } \nu(a_k)-(\a k+\b)\nu(h)=k,\, \forall\,k\geq k_0
\]
We denote by $\PP_h\lb \x\rb$ the subring of $\PP_h (\!(\x)\!)$ of elements $A$ such that $k_0$ belongs to $\mathbb{Z}_{\geq 0}$, and we denote by $\PP_h\nl\x\nr$ the subring of $\PP_h\lb \x\rb$ of elements $A$ such that 
\[\sum_{k\geqslant k_0} a_k(\x)\in \KK\nl\x\nr.
\]
When $\g$ is an integral homogeneous element, we denote by $\PP_{h}\lb \x,\g \rb$ the subring of $\wdh  V_{\nu,\g}$, whose elements $\xi$ are of the form:
\[
\xi=\sum_{k=0}^{d-1} A_k(\x) \gamma^k, \quad \text{ where } A_k \in \PP_h (\!(\x)\!) \text{ and } \nu(A_{k}(\x)\g^{k})\geq 0, \, \, \, k=0,  \ldots,d-1.
\]
\end{definition}

\begin{remark}
Lemma \ref{lemma:TempExt} below shows that if $A\in \PP_h\lb\x\rb$, the fact that $A\in \PP_h\nl\x\nr$ is independent of the presentation of $A$, that is, $\PP_h\nl\x\nr$ is well-defined. This observation greatly simplifies \cite[Prop 5.13]{BCR}, which relied in complex analysis.
\end{remark}

The next two results have been proven (in greater generality) in \cite{R}, but we refer the reader to \cite{BCR} where the statement is given when $\KK=\C$, but whose proof remains valid in the case of a general characteristic zero field.

\begin{theorem}[{Newton-Puiseux-Eisenstein, cf. \cite[Th 5.8]{BCR}}]\label{thm:NewtonPuiseux}
Let $\KK$ be a characteristic zero field and let $P(\x,y) \in \KK\lb \x\rb [y]$ be a monic polynomial. There exists an integral homogeneous element $\g$, and a homogeneous polynomial $h(\x)$, such that $P(\x,y)$ factors as a product of degree $1$ monic polynomials in $y$ with coefficients in $\PP_h\lb \x, \g\rb$.
\end{theorem}

The following result is a convenient reformulation of Theorem \ref{thm:NewtonPuiseux}:
\begin{corollary}[{Newton-Puiseux-Eisenstein factorization, cf. \cite[Cor 5.9]{BCR}}]\label{cr:NewtonPuiseux}
Let $\KK$ be a characteristic zero field and let $P\in\KK\lb \x\rb[y]$ be a monic polynomial. Then, there is a homogenous polynomial $h$ and  integral homogenous elements $\g_{i,j}$, such that $P$ can be written as 
\begin{equation}\label{eq:key_factorization}
P(\x,y)=\prod\limits_{i=1}^s Q_i, \quad \text{ and } \quad  Q_i=\prod\limits_{j=1}^{r_i} (y-\xi_i(\x,\g_{i,j}))
\end{equation}
where 
\begin{enumerate}[label=(\roman*)]
\item the $Q_i\in \PP_h\lb \x\rb [y]$ are irreducible in $\wdh V_\nu[y]$, 
\item for every $i$, there are $A_{i,k}(\x)\in\PP_h(\!(\x)\!)$, for $0\leq k\leq k_i$ such that
$$\xi_i(\x,\g_{i,j})=\sum_{k=0}^{k_i}A_{i,k}(\x)\g_{i,j}^k\in\PP_{h}\lb \x,\g_{i,j} \rb$$
\item
for every $i$, the $\g_{i,j}$ are distinct conjugates of an homogeneous element $\g_i$, that is, roots of its minimal polynomial $\Gamma_i$ over $\KK(\x)$.
\end{enumerate}
\end{corollary}


\subsection{Blowings-up and the geometric setting}

In what follows, we use algebraic-geometry methods concerning blowings-up $\sigma: N' \lgw N $, where $N$ will stand for some affine space over $\KK$ (the precise meaning of this statement will be clarified in this subsection). Nevertheless, and in contrast to usual algebraic and analytic geometry, we do not have access ,as far as we know, to a theory of varieties and sheaves valid for W-temperate families. We do not have the ambition to develop such a general theory in here, but rather to introduce the minimal set of definitions which are necessary for this work. In particular, we will greatly exploit the fact that we only need to work over $\KK \nl \x \nr$, where $\x=(x_1,x_2)$ stands for \emph{two indeterminates}, in order to avoid a more technical discussion.

Let us start by fixing a set of indeterminate $\x = (x_1,\ldots,x_n)$, and a W-temperate ring $\KK \nl \x \nr$. In what follows, we will often need to change indeterminate:

\begin{definition}[Temperate automorphism]
Let $\phi$ be a $\KK$-automorphism of the ring of power series $\KK\lb \x\rb$. We say that $\phi$ is \emph{temperate} if $\phi(\KK\nl \x\nr)\subset\KK\nl \x\nr$.
\end{definition}

\begin{lemma}
A $\KK$-automorphism $\phi$ given by series $\phi(x_i)\in\KK\lb \x\rb$ is temperate if and only if for every $i$ we have $\phi(x_i)\in\KK\nl \x\nr$. In this case  $\phi^{-1}$ is also temperate. In particular, when $\phi$ is temperate, we have $\phi(\KK\nl \x\nr)=\KK\nl \x\nr$.
\end{lemma}
\begin{proof}
The condition is necessary by definition, and sufficient from the fact that $\KK\nl \x\nr$ is stable by composition. Now $\phi^{-1}$ is also temperate since $\KK\nl \x\nr$ satisfies the implicit function Theorem, cf. Proposition \ref{rk:Wtemp}\,\ref{prop:hensel}.
\end{proof}

It follows directly from this lemma that any $\KK$-linear automorphism in $\x$ is a temperate automorphism. We are ready to introduce the notion of temperate coordinate systems:

\begin{definition}[Temperate coordinates]
Let $\KK\nl \x \nr$ be a temperate ring. A system of parameters $\widetilde{\x}$ of the ring $\KK\lb \x\rb$ is said to be \emph{temperate} if $\widetilde{\x}$ is obtained from $\x$ by a temperate $\KK$-automorphism. A system of parameters $\widehat{\x}$ of $\KK\lb \x\rb$ which is not temperate will be called formal.

We will denote by $\mathcal{O}$ the intrinsic ring of temperate power series associated to $\KK\nl \x \nr$ up to temperate automorphisms, that is, $\mathcal{O}$ denotes $\KK\nl \x \nr$, for some temperate coordinate system $\x$, and is isomorphic to $\KK\nl \widetilde{\x} \nr$ for any temperate coordinate system $\widetilde{\x}$.
\end{definition}

We now specialize to the case that $n=2$. Let $\KK \nl x_1,x_2 \nr$ be a temperate ring and $N_0 = \wdh{\mathbb A}^2_\KK$ be the  affine scheme associated to the complete local ring $\KK\lb x_1,x_2\rb$. We denote by $\mathcal{O}_0$ and $\widehat{\mathcal{O}}_{0}$ the rings of temperate and formal power series at $0$. We consider the formal blowing-up of the origin:
\[
\sigma : (N_1, E_1) \lgw (N_0,0)
\]
where $\s^{-1}(0) = E$ is called the exceptional divisor. Given any closed point $\pb \in E$, we can localize $\sigma$ to $\pb$ in order to obtain a morphism between local rings $\s_{\pb}^* : \widehat{\mathcal{O}}_{0} \lgw\widehat{\mathcal{O}}_{\pb}$, where $\widehat{\mathcal{O}}_{\pb}$ stands for the local ring of formal power series at $\pb$. Now, apart from a $\KK$-linear change of indeterminacy in $\x$ (which is a temperate change of coordinates), we may suppose that $\pb$ is  the origin of the $x_1$-chart of the blowing-up, that is, there exists a system of parameters $\vb = (v_1,v_2)$ of $\widehat{\mathcal{O}}_{\pb}$ such that $\s_{\pb}^* : \KK \lb \x \rb \lgw\KK \lb \vb \rb$ is given by
\[
(x_1,x_2) \lgm (v_1,v_1v_2).
\]
We note that the ideal of $E$ is generated by $v_1$ in this chart.

\begin{definition}
Following the above construction, we say that $\vb = (v_1,v_2)$ is a system of temperate coordinates at $\pb$. In particular $\sigma_{\pb}^*$ induces a morphism:
\[
\s_{\pb}^* : \KK \nl \x \nr \lgw \KK \nl \vb \nr.
\]
\end{definition}

The next lemma shows that this definition is consistent with temperate changes of coordinates, allowing us to write:
\[
\s_{\pb}^{\ast} : \mathcal{O}_0 \lgw \mathcal{O}_{\pb}.
\]
\begin{lemma}
Let $\widetilde{\x}=(\widetilde{x}_1,\widetilde{x}_2)$ be a different system of temperate coordinates at $0$, that is, there exists a temperate authomorphism $\phi : \KK \nl \widetilde{\x} \nr \lgw\KK \nl \x \nr$. Suppose that there exists a system of parameters $\widetilde{\vb}=(\widetilde{v}_1,\widetilde{v}_2)$ of $\widehat{\mathcal{O}}_{\pb}$ such that:
\[
(\widetilde{x}_1,\widetilde{x}_2) \lgm (\widetilde{v}_1,\widetilde{v}_1\widetilde{v}_2).
\]
Then $\widetilde{\vb}$ is a system of temperate coordinates, that is, there exists a temperate automorphism $\psi : \KK \nl \widetilde{\vb} \nr \lgw\KK \nl \vb \nr$.
\end{lemma}
\begin{proof}
Let $\phi(\widetilde{x}_1) = \phi_1(x_1,x_2)$ and $\phi(\widetilde{x}_2) = \phi_2(x_1,x_2)$. From the assumption
\begin{equation}\label{eq:ChangeCoordinates}
\widetilde{v}_1 = \phi_1(v_1,v_1v_2), \quad \widetilde{v}_2 = \frac{\phi_2(v_1,v_1v_2)}{\phi_1(v_1,v_1v_2)},
\end{equation}
and from usual formal algebraic geometry, we know that \eqref{eq:ChangeCoordinates} defines an authomorphism of $\widehat{\mathcal{O}}_{\pb}$. Let us show that this automorphism is temperate. We consider the Taylor expansion of $\phi_1$ and $\phi_2$ in order to get:
\[
\phi_1(v_1,v_1v_2) = v_1 \left(a_{1,1} + a_{1,2}v_2 + v_1\Phi_1 \right), \quad \phi_2(v_1,v_1v_2) = v_1 \left(a_{2,1} + a_{2,2}v_2 + v_1\Phi_2 \right)  
\]
where the $\KK$-matrix $A = [a_{i,j}]$ is invertible and $\Phi_1$ and $\Phi_2$ are temperate functions by Proposition \ref{rk:Wtemp} \ref{rk:division}. Therefore:
\[
\widetilde{v}_2 = \frac{a_{2,1} + a_{2,2}v_2 + v_1\Phi_2}{a_{1,1} + a_{1,2}v_2 + v_1\Phi_1}
\]
and we conclude that $a_{1,1}\neq 0$ and $a_{2,1} =0$. The result is now immediate from the implicit function Theorem, cf. Proposition \ref{rk:Wtemp} \ref{prop:hensel}.
\end{proof}

In what follows, we will consider sequences of point blowings-up
\[
\xymatrix{ (N_r,F_r) \ar[r]^{ \quad \,\, \tg_r}  &  \cdots   \ar[r]^{ \tg_2 \quad \, \,} &  (N_1,F_1)  \ar[r]^{\tg_1 \quad \quad}& (N_0,0)= (\wdh{\mathbb A}^2_\KK,0)}
\]
and it will be convenient to fix notation. We set $\s=\s_1\circ\cdots\circ \s_r$ and, for every $j\in\{1,\ldots, r\}$, $F_j$ is a simple normal crossing divisor that can be decomposed as
$$
F_j=F_j^{(1)}\cup F_j^{(2)}\cup \cdots \cup F_j^{(j)}
$$
where $F_j^{(k)}$ is the strict transform of $F_{j-1}^{(k)}$ (when $k<j$) and $F_j^{(j)}$ is the exceptional divisor of $\s_j$. Now, fixed a temperate ring $\mathcal{O} = \KK \nl x_1,x_2\nr$ at $0$, the formal morphism $\s$ can be localized at every point $\pb \in F_r$ in order to generate a morphism between temperate rings, that is, there are system of parameters $\vb = (v_1,v_2)$ of $\widehat{\mathcal{O}}_{\pb}$ such that $\s_{\pb}^{\ast} : \KK \nl \x \nr \lgw\KK \nl \vb \nr$ is well-defined and can be written $\s_{\pb}^{\ast} : \mathcal{O}_0 \lgw\mathcal{O}_{\pb}$. 

\begin{remark}\label{rk:NiceCoordinatesF1}
If $\pb \in F_r^{(1)}$ then, from usual combinatorial considerations about blowings-up, we may further suppose that $\s_{\pb}^{\ast} : \KK \nl \x \nr \lgw \KK \nl \vb \nr$ is given by:
\[
(x_1,x_2) \lgm (v_1v_2^c, v_1 v_2^{c+1})
\]
for some natural number $0\leq c\leq r$.
\end{remark}

\subsection{Blowings-up and Projective rings}
We present in this subsection different results about the behavior of projective series and temperate projective series under blowing-up, which will be most useful in the sequel. 

\begin{definition}
Let $A\in \PP_h \lb \x\rb$ and $\s: (N_r,F_r) \lgw (N_0,0)$ a sequence of point blowings-up. We say that $A$ \emph{extends at a point $\pb \in F_r$} if $A_\pb:=\sigma_\pb^*(A)$ belongs to $\widehat{\mathcal{O}}_{\pb}$. Furthermore, we say that $A$  \emph{extends temperately} if $A_\pb\in \mathcal{O}_{\pb}$, where we recall that $\mathcal{O}_{\pb}$ stands for the ring of W-temperate functions at $\pb$.
\end{definition}

The next Lemma is a generalization of \cite[Proposition 5.13 and Lemma 5.14]{BCR} for W-temperate rings. Note that the proof given in \cite{BCR} relies in complex analysis, cf. \cite[$\S\S$5.3]{BCR}, and does not adapt in a trivial way to W-temperate rings, so we provide a new commutative algebra argument:

\begin{lemma}[{cf. \cite[Proposition 5.13 and Lemma 5.14]{BCR}}]\label{lemma:TempExt}
Let $A\in \PP_h\lb \x\rb$ and let $\s: (N_r,F_r) \lgw (N_0,0)$ be a sequence of point blowings-up. Let $\pb\in F_r^{(1)}$ be such that $A$ extends at $\pb$ (that is the case for instance when $\pb$ does not belong to the strict transform of $h=0$ or the in the intersection with $F_r^{(j)}$ for some $j>1$). Then $A\in \POx$ if and only if $A_\pb\in \mathcal{O}_{\pb}$. In particular, $\PP_h\nl \x\nr \cap \KK \lb \x \rb = \KK\nl \x \nr $.
\end{lemma}

\begin{proof}
Let $A \in \PP_h\lb \x\rb$ and fix a point $\pb\in F_r^{(1)}$. By definition \ref{def:ProjRings} and the local expressions of blowings-up given in Remark \ref{rk:NiceCoordinatesF1}, we can write:
\[
A=\sum\limits_{k\geqslant k_0} \frac{a_k(\x)}{h(\x)^{\alpha k+\beta}}, \quad \text{and} \quad A_\pb=\sum\limits_{k\geqslant k_0} (v_1v_2^{c})^k\frac{a_k(1,v_2)}{h(1,v_2)^{\alpha k+\beta}}.
\]
Denote by $d$ the degree of $h$, and consider:
\[
\widetilde{A} = \sum\limits_{k\geqslant k_0} a_k(\x) , \quad \text{and} \quad \widetilde{A}_{\pb} = \sigma_{\pb}^{\ast}(\widetilde{A})  = \sum\limits_{k\geqslant k_0} (v_1v_2^{c})^{(\a k  + \b)d + k}   a_k(1,v_2). 
\]
Let us define the following auxiliary function:
\[
B(\w) :=\sum\limits_{k\geqslant k_0} (w_1w_2^{c})^k a_k(1,w_2) \in \KK\lb\\w \rb.
\]
Now, writing $h(1,v_2)=v_2^mg(v_2)$, where $g$ is a unit and $ m \in \mathbb{N}$, we have:
\[
\begin{aligned}
\widetilde{A}_{\pb}(\vb) &=  (v_1v_2^c)^{\b d} B(v_1^{\a+1}v_2^{c \a d},v_2) \\
B(\w) &=  w_2^{m\b} g(w_2)^{\b} A_\pb(w_1 w_2^{m \a} g(w_2)^\a,w_2).
\end{aligned}
\]
Since being temperate is closed by division by a coordinate, ramification and local blowing-down (see Proposition \ref{rk:Wtemp} \ref{rk:ramif}, \ref{rk:division} and Definition \ref{temperate_fam}), and $A_{\pb} \in \KK\lb \vb \rb$ by hypothesis, we conclude that
\[
\widetilde{A}_{\pb}(\vb) \in \KK \nl \vb \nr \iff B(\w) \in \KK \nl \w \nr \iff A_{\pb}(\vb) \in \KK \nl \vb \nr,
\]
finishing the proof.
\end{proof}

We conclude this subsection with a useful characterization of projective series which are not formal power series:

\begin{lemma}\label{lem:inductive}
Let $A\in\PP_h\lb\x\rb\setminus \KK\lb \x\rb$ and consider a point blowing-up $\s$ centered at the origin. There exists a point $\pb\in \s^{-1}(0)$ such that $A_{\pb}=\sigma^{\ast}_{\pb}(A)$ is not a power series, that is, $\sigma^{\ast}_{\pb}(A) \notin \widehat{\mathcal{O}}_{\pb}$.
\end{lemma}

\begin{proof}
Let $A\in\PP_h\lb\x\rb\setminus \KK\lb \x\rb$; from definition \ref{def:ProjRings} we may write
\[
A=\sum_{k\in\N}\frac{a_k(\x)}{b_k(\x)}
\]
where the $a_k$ and $b_k$ are homogeneous polynomials in $\KK[\x]$ such that $\deg(a_k)-\deg(b_k)=k$ and $\gcd(a_k,b_k)=1$. By hypothesis, there exists $k_0$ such that $b_{k_0}(\x)$ is not a constant polynomial. Apart from a $\KK$-linear change of coordinates in $\x$, we may furthermore suppose that $b_{k_0}(1,0) =0$. It follows that after the local blowing-up $\sigma:(x_1,x_2)\lgm (v_2,v_1v_2)$ we obtain
\[
\sigma^{\ast}_\pb(A) = \sum_{k\in\N} v_1^k \frac{a_k(1,v_2)}{b_k(1,v_2)},
\]
this expression has a pole in the term $k_0$, and we conclude easily.
\end{proof}

\subsection{Extension along the exceptional divisor}

We introduce the notion of Laurent series with support in a strongly convex cone, and we refer the reader to \cite{AI} for extra details.

\begin{definition}
Let $\om\in(\R_{>0})^n$ be a vector whose coordinates are $\Q$-linearly independent. This vector defines a total order on the set of monomials by setting
$$
\x^\a\preceq\x^\b \text{ if } \a\cdot\om\leq\b\cdot\om.
$$
Let $\Sigma$ be a strongly rational cone. We say that $\Sigma$ is $\om$-positive if  $s\cdot\om>0$ for every $s\in\Sigma\setminus\{0\}$; under this hypothesis, $\Sigma\cap\Z^n$ and $\Sigma\cap\frac1q\Z^n$ for $q\in\N^*$ are well-ordered for $\preceq$, and $(\R_{\geq 0})^\subset \Sigma$ because $\om\in(\R_{>0})^n$.

Assume that $\Sigma$ is $\om$-positive strongly rational cone. We denote by $\KK\lb \Sigma\rb$ (resp. $\KK\lb \Sigma\cap\frac1q\Z^n\rb$ for $q\in\N^*$) the set of Laurent series with support in $\Sigma\cap\Z^n$ (resp. with support in $\Sigma\cap\frac1q\Z^n$). Since $\Sigma\cap\Z^n$ and $\Sigma\cap\frac1q\Z^n$ are well-ordered for $\preceq$, they are rings containing respectively $\KK\lb \x\rb$ and $\KK\lb\x^{1/q}\rb$. These rings are commutative integral domains, and we denote by $\KK(\!(\Sigma)\!)$ and $\KK(\!(\Sigma\cap\frac1q\Z^n)\!)$ their respective fraction fields. 
\end{definition}

The next result is a generalization of \cite[Theorem 5.16]{BCR} for W-temperate rings. Once again, the proof given in \cite{BCR} relies in complex analysis, cf. \cite[$\S\S$5.4]{BCR}, so we can not adapt it in a trivial way to W-temperate ring. Instead, we provide a new commutative algebra argument, which greatly simplifies the proof:

\begin{theorem}[{cf. \cite[Theorem 5.16]{BCR}}]\label{thm:extension}
Let $P\in \KK\lb\x\rb [y]$ be a monic reduced polynomial, and let $Q$ be an irreducible factor of $P$ in some $\PP_h\lb \x\rb [y]$ for a convenient $h\in \KK[\x]$ as in Corollary \ref{cr:NewtonPuiseux}. Let $\s: (N_r,F_r) \lgw (N_0,0)$ be a sequence of point blowings-up such that $\sigma^*(\Delta_P)$ is everywhere monomial, that is, at any point $\pb$ there exist (non necessarily temperate) coordinates $\widehat{\vb}$ such that 
$$
\sigma^*(\Delta_P)=\widehat{\vb}^\a\times\unit.
$$
Then $Q$ extends at every point $\pb'\in F_r^{(1)}$.
\end{theorem}

\begin{proof}
Let $\pb\in F_r^{(1)}$. From Remark \ref{rk:NiceCoordinatesF1}, there are coordinates $\widehat{\vb} = (\widehat{v}_1,\widehat{v}_2)$ centered at $\pb'$ and $c\in\N$ such that
$$
(x_1,x_2)=(\widehat{v}_1\widehat{v}_2^c,\widehat{v}_1\widehat{v}_2^{c+1}).
$$
Let $A$ be a coefficient of $Q$. By definition \ref{def:ProjRings}, and by writing $h(1,\widehat{v}_2)=\widehat{v}_2^mg(\widehat{v}_2)$ where $g$ is a unit and $m \in \mathbb{N}$, we have
\begin{equation}\label{supp1}
A=\sum_{k\in\N}\frac{a_k(\x)}{h(\x)^{\a k+\b}} \quad \text{so that} \quad  A_\pb= \widehat{v}_2^{-m\b} \sum_{k\in\N}\widehat{v}_1^k\widehat{v}_2^{k(c-m\a)}\frac{a_k(1,\widehat{v}_2)}{g(1,\widehat{v}_2)^{\a k+\b}}.
\end{equation}
Note that the series $\widehat{v}_2^{m\b}A_\pb$ has support in a translation of the strongly convex cone $\Sigma$ generated by the vectors $(0,1)$ and $(1, \min \{0,c-m\a\})$, thus $A_\pb$ belongs to $\KK(\!( \Sigma)\!)$. We conclude that $Q_{\pb}=\s_\pb^*(Q)$ is a factor of $P_{\pb}=\s_\pb^*(P)$ in $\KK(\!(\Sigma)\!)[y]$.

Now, by the Abhyankar-Jung Theorem for formal power series, the roots of $P_{\pb}$ can be written as Puiseux power series in $\KK\lb \widehat{v}_1^{1/q},\widehat{v}_2^{1/q}\rb\subset \mathcal \KK\lb\Sigma\cap\frac1q\Z^2\rb$ for some $q\in\N^*$. Since $\KK(\!(\Sigma\cap\frac1q\Z^n)\!)$ is a field, we conclude that $Q_{\pb}$ splits in $\KK(\!(\Sigma\cap\frac1q\Z^n)\!)[y]$ and its roots are in $\KK\lb \widehat{v}_1^{1/q},\widehat{v}_2^{1/q}\rb$. By \eqref{supp1}, we conclude that $Q_{\pb}\in\KK\lb \widehat{\vb} \rb$.
\end{proof}

We are ready to prove the main result of this subsection, which generalizes \cite[Theorem 5.18]{BCR} for W-temperate rings. We highlight that this is the only point where Definition \ref{W-fam} \ref{axiom:last} intervenes:

\begin{theorem}[{cf. \cite[Theorem 5.18]{BCR}}]\label{local_to_semiglobal}
Let $P\in\KK\lb x_1,x_2\rb[y]$ be a monic reduced polynomial, and $h$ be a homogeneous polynomial for which Theorem \ref{thm:NewtonPuiseux} is satisfied. Let $\s: (N_r,F_r) \lgw(N_0,0)$ be a sequence of point blowings-up. Suppose:
\begin{itemize}
\item   At every point $\pb\in F_r^{(1)}$, the pulled-back discriminant $ \s_\pb^*(\Delta_P)$ is monomial;
\item There exists $\pb_0\in F_r^{(1)}$ such that $P_{\pb_0}=\s_{\pb_0}^*(P)$ admits a factor in $\mathcal{O}_{\pb_0}$.
\end{itemize}
Then $P$ admits a non-constant factor $Q\in\POx[y]$, such that either $P/Q$ is constant, or $\s_\pb^*(P/Q)$ admits no non-constant temperate factor for all $\pb\in F^{(1)}_r$.
\end{theorem}

\begin{proof}
Consider the factorization $P=\prod_{i=1}^s Q_i$ given in Corollary \ref{cr:NewtonPuiseux}, where the $Q_i$ belong to some $\PP_h\lb \x\rb [y]$. It follows from Remark \ref{rk:NiceCoordinatesF1} that there exists temperate coordinates $\vb =(v_1,v_2)$ centered at $\pb_0$ such that $\s_{\pb_0}^{\ast}$ is locally given by 
\(
(x_1,x_2) \mapsto (v_1v_2^c,v_1v_2^{c+1}),
\) 
so that we get:
\[
P_{\pb_0}=\prod_{i=1}^s \s_{\pb_0}^*(Q_i)
\]
where the $\s_{\pb_0}^*(Q_i) \in \KK \lb \vb\rb [y]$ have formal power series coefficients according to Theorem \ref{thm:extension}.  Moreover, since $\KK\lb\x\rb$ is a UFD,  for some $i_0$, $\s_{\pb_0}^*(Q_{i_0})$ has a common factor with a polynomial in $\OOv[y]$. Therefore, by Corollary \ref{cor:Artin}, $\s_{\pb_0}^*(Q_{i_0})$ has a non trivial divisor $R\in \OOv[y]$ that is monic in $y$.

We claim that $\s_{\pb_0}^*(Q_{i_0})$ has its coefficients in $\OOv$. Note that the Theorem immediately follows from the Claim applied to every polynomial $Q_i$ having a temperate factor at some point of $F_r^{(1)}$. Let us prove the Claim. For simplicity we denote $Q_{i_0}$ by $Q$, and $\s_{\pb_0}^*(Q_{i_0})$ by $Q_{\pb_0}$. Now, we may suppose without loss of generality that the discriminant of $P$ is monomial in respect to the temperate coordinate system $(v_1,v_2)$. Indeed, up to making a blowing-up with center $\pb_0$, we may suppose that the discriminant of $P$ is monomial in respect to the temperate coordinate system $(v_1,v_2)$ by considering, for example, the point $\pc_0 = F_{r+1}^{(1)} \cap F_{r+1}^{(r+1)}$; we note that if we show that $\sigma_{\pc_0}^{\ast}(Q)$ is W-temperate, then so is $Q_{\pb_0}$ by Def. \ref{temperate_fam} \ref{axiom:monomialcomposition}.

At the one hand, we may apply Abhyankhar-Jung Theorem for formal power series in order to show that $Q_{\pb_0}$ splits in $\KK\lb \vb^{1/q}\rb$ for some $q\in\N$, that is
\(
Q_{\pb_0} = \prod (y - \psi_j)
\) where $\psi_j \in \KK\lb \vb^{1/q}\rb$. Furthermore, we may apply the temperate Abhyankhar-Jung Theorem \ref{A-J} to the temperate factor of $Q_{\pb_0}$, in order to conclude that one of these roots is temperate, say, $\psi_1 \in \KK\nl \vb^{1/q}\nr$. At the other hand, by Theorem \ref{thm:NewtonPuiseux}, the roots of $Q$ belong to a ring $\PP_h\lb\x,\g_1\rb$, where $\g_1$ is an integral homogeneous element.  Let $\G(\x,z)\in\KK[\x,z]$ be the irreducible $\omega$-weighted homogeneous polynomial having $\g_1$ as a root (see Definition \ref{homogeneous_element}) and that is monic in $z$. By Corollary \ref{cr:NewtonPuiseux}, the roots of $Q$ are given by $\xi(\x,\g)$ where $\g$ runs over the roots of $\G(\x,z)$, that is, $Q = \prod(y-\xi(\x,\g'))$ where the product is taken over every root $\g'$ of $\G$. In what follows, we perform a detailed study on how roots of $Q$ in $\PP_h\lb\x,\g_1\rb[y]$ transform by the blowing-up in order to compare them to the temperate root of $Q_{\pb_0}$ in $\KK \lb \vb^{1/q} \rb $.

We start by describing how the roots $\gamma'$ of $\G$ transform by $\s$. Consider
\[
\G(\x,z)=z^d+\sum_{i=1}^df_i(\x)z^{d-i}
\]
where the $f_i(\x)$ are homogeneous polynomials of degree $\om i$. Since $\KK$ is algebraically closed, we may suppose that $\om>0$ (otherwise $Q$ is a degree one polynomial, and the Claim is trivial), that is, $\G(\x,z)$ is a Weierstrass polynomial in $z$. We write $\om=p/e$ with $\gcd(p,e)=1$, and we note that $f_i=0$ if $e$ does not divide $i$. Furthermore, because $\G$ is irreducible, $f_d\neq0$, hence, $e$ divides $d$. We have
\[
\begin{aligned}
\G(v_1v_2^c,v_1v_2^{c+1},v_1^\om z)&=z^d+\sum_{i=1}^df_i(v_1v_2^c,v_1v_2^{c+1})(v_1^\om z)^{d-i}\\
&=v_1^{d\om}\left(z^d+\sum_{j=1}^{d/e}v_2^{cpj}f_{ej}(1,v_2)z^{d-ej}\right)
\end{aligned}
\]
and we set
\[
\ovl \G(v_2,z)=z^d+\sum_{j=1}^{d/e}v_2^{cpj}f_{ej}(1,v_2)z^{d-ej}\in\KK[v_2,z^e]\subset\KK[v_2,z].
\]
Note that $\wdt\g$ is a root of $\sigma^{\ast}(\G) = \G(v_1v_2^c,v_1v_2^{c+1}, z)$ if and only if $\wdt\g=v_1^\om\ovl\g$ where $\ovl\g$ is a root of $\ovl\G(v_2,z)$. Now, let us remark that $\ovl\G$ is irreducible in $\KK[v_2,z^e]$. Indeed, if $\ovl \G=\ovl\G_1\ovl\G_2$ where $\ovl \G_i\in\KK[v_2,z^e]$ have positive degree $\ell_i$ in $z^e$, we set
$\G_i'(v_1,v_2,z^e):=v_1^{\ell_ie\om}\ovl \G_i$ for $i=1$, $2$.
 Then we would have
$$
\G(\x,z)=\ovl \G_1\left(\frac{x_1^{c+1}}{y^c},\frac{x_2}{x_1},\frac{x_2^{cp}}{x_1^{(c+1)p}}z^e\right)\ovl \G_2\left(\frac{x_1^{c+1}}{y^c},\frac{x_2}{x_1},\frac{x_2^{cp}}{x_1^{(c+1)p}}z^e\right)
$$
contradicting the irreducibility of $\G(\x,z)$. In particular, this implies that the irreducible factors of $\ovl\G(v_2,z)$ are conjugates up to multiplication of $z$ by a $e$-th root of unity. This means that we may write
$$
\ovl\G(v_2,z)=\prod\ovl \G_{\eta}(v_2,z)
$$
where $\eta$ runs through a subgroup $H$ of the group of the $e$-th root of unity and the $\ovl\G_{\eta}(v_2,z)$ are irreducible (monic in $z$) polynomials, such that
$$
\ovl\G_{\eta}(v_2,z)=\ovl\G_1(v_2,\eta z).
$$
It follows that we may parametrize all roots of $\ovl\G$ by $\ovl\g_{i,\eta}$ for $1=1,\ldots,d/e'$ and $\eta \in H$, where $e' = |H|$ and $\ovl\g_{i,\eta} = \eta \cdot \ovl\g_{i,1}$. We may index the roots of $\G$, therefore, by $\g_{i,\eta}$ in such a way that $\sigma^{\ast}_{\pb_0}(\g_{i,\eta}) = \wdt\g_{i,\eta} = v_1^{\omega} \ovl\g_{i,\eta}$ are the roots of $\sigma^{\ast}(\G) = \G(v_1v_2^c,v_1v_2^{c+1}, z)$. We fix the convention that $\fract{\sigma^{\ast}_{\pb_0}(\g_1)}{ v_1^\om}=\ovl\g_{1,1}$ and, more generally, that $\fract{\sigma^{\ast}_{\pb_0}(\g_i)}{v_1^\om}=\ovl \g_{i,1}$ are all the roots of $\ovl \G_1$. Next, by Newton-Puiseux Theorem, we can write the roots of $\ovl\G(v_2,z)$ as Puiseux series in $\KK\langle v_2^{1/q}\rangle$, even if it means replacing $q$ by a larger integer. 

Now, we use the normal form given by Definition \ref{def:ProjRings}, in order to write
\[
\xi(\x,\g_{i,\eta})=\sum_{j=0}^{d-1}A_j(\x)\g_{i,\eta}^j \text{ with } A_j(\x)=\sum_{k\geq k_j}\frac{a_{k,j}(\x)}{h^{\a_jk+\b_j}(\x)}
\]
where the $a_{k,j}(\x)$ are homogeneous polynomials. Since there are only  finitely many  $j$, apart from multiplying the numerators and the denominators of the coefficients of $A_j(\x)$
by a power of $h(\x)$, we may assume that the $\a_j$ (resp. the $\b_j$) are all independent of $j$ and equal to some integer $\a$ (resp. $\b$). Note that
\begin{equation}\label{key_dev}\begin{split}
 \s_{\pb_0}^*(\xi(\x,\g_{i,\eta}))
 =\sum_{j=0}^{d-1}\s_{\pb_0}^*(A_j(\x))\wdt\g_{i,\eta}^j&=\sum_{j=0}^{d-1} \ovl \g_{i,\eta}^j\sum_{k\geq k_j} v_1^{k+\om j} v_2^{ck}\frac{a_{k,j}(1,v_2)}{h^{\a k+\b}(1,v_2)}\\ 
 &=\sum_{k\in\frac1e\N}\frac{v_1^k}{h(1,v_2)^{\a k+\b}}b_k(v_2,\ovl\g_{i,\eta})
\end{split}\end{equation}
where  $b_k\in \KK[v_2,z]$ with $\deg_z(b_k)\leq d-1$. We remark that $\deg_{v_2}(b_k)$ is bounded by a linear function in $k$ because, for each $j$, $\deg_\x(a_{k,j}(\x))$ is bounded by a linear function in $k$. 

We note that we can write $h(1,v_2)=v_2^mg(v_2)$ for some unit $g(v_2)$, and some $m\in \N$. Therefore, as already shown in the proof of Theorem \ref{thm:extension}, the series $v_2^{m\b}(A_j)_\pb$ are Laurent series with support in the  strongly convex cone $\Sigma$ generated by the vectors $(0,1)$ and $(1,\min\{0,c-m\a\}$. Therefore, if we identify the $\ovl \g_{i,\eta}$ with their expansions as  Puiseux series of $\KK\langle v_1^{1/q}\rangle$, we have that 
$$
v_2^{m\b}\sigma_{\pb_0}(\xi(\x,\g_{i,\eta})) \in \KK \left\llbracket \Sigma \cap \frac{1}{eq}\Z^2  \right\rrbracket .
$$
Since $\KK\nl \vb \nr  \subset \KK\lb \Sigma \cap \frac{1}{eq}\Z^2 \rb  \subset \KK (\!( \Sigma \cap \frac{1}{eq}\Z^2 )\!)$, and $\KK (\!( \Sigma \cap \frac{1}{eq}\Z^2 )\!)[y]$ is a UFD, we conclude that the set of roots $\s_{\pb_0}^*(\xi(\x,\g_{i,\eta}))$ and $\psi_j$ of $Q_{\pb_0}$ must coincide when we expand the $\ovl\g_{i,\eta}$ as Puiseux series. From now, the $\ovl\g_{i,\eta}\in\KK\langle v_2^{1/q}\rangle$. We set $\psi_{i,\eta} = \s_{\pb_0}^*(\xi(\x,\g_{i,\eta}))$; note that $\psi_{i,\eta} \in \KK \lb v_1^{1/e},v_2^{1/q} \rb$ for every $i$ and $\eta$ and, apart from re-indexing, we have $\psi_{1,1} \in \KK \nl v_1^{1/e},v_2^{1/q} \nr$.

Next, note that for every $e$-th root of unity $\eta$, there exists a $e$-th root of unity $\widetilde{\eta}$ such that $ \widetilde{\eta}^p = \eta$ since $\gcd(e,p)=1$, so that:
\begin{align*}
\psi_{1,\eta}(v_1^{1/e},v_2)& = \wdh\s_{\pb_0}^*(\xi(\x,\eta\, \g_1))=\sum_{j=0}^{d-1} (\eta\,\ovl \g_1)^j\sum_{k\geq k_j} v_1^{k+\om j} v_2^{ck}\frac{a_{k,j}(1,v_2)}{h^{\a k+\b}(1,v_2)}\\ 
 &=\sum_{j=0}^{d-1} \ovl \g_1^j\sum_{k\geq k_j}  (\widetilde{\eta}\, v_1^{1/e})^{ek+ p j} v_2^{ck}\frac{a_{k,j}(1,v_2)}{h^{\a k+\b}(1,v_2)} = \psi_{1,1}(\widetilde{\eta} \, v_1^{1/e},v_2),
\end{align*}
so that $\psi_{1,\eta} \in \KK\nl v_1^{1/e},v_2\nr $ for every $e$-th root of unity $\eta$. More generally, this argument shows that:
\[
\forall i,\qquad \psi_{i,1} \in \KK\nl v_1^{1/e},v_2\nr  \implies \psi_{i,\eta} \in \KK\nl v_1^{1/e},v_2\nr,
\]
for every $\eta\in H$. We are, therefore, reduced to show that $\psi_{i,1} = \s_{\pb_0}^*(\xi(\x,\g_i)) \in \KK\nl v_1^{1/e},v_2^{1/q}\nr$ for all $i = 1,\ldots,d/e'$, where we recall that the $\ovl\g_i$ are the roots of the irreducible polynomial $\ovl{\G}_1$. 
Now, we introduce the auxiliary function
\begin{equation}\label{fct-B}
B(\w,z):=\sum\limits_{k\in \frac1e\N} w_1^{ek} b_{k}(w_2^{q},z)\in\KK\lb w_1,w_2^q\rb[z].
\end{equation}
where $\deg_{w_2}(b_{k})$ is bounded by a linear function in $k$. Since $\psi_{1,1}(\vb) \in \KK\nl v_1^{1/e},v_2\nr$, \eqref{key_dev} and \eqref{fct-B} imply that
\[
B(\w, \ovl\g_1(w_2)) = w_2^{\beta m q} g(w_2^q)^{\beta} \cdot \psi_{1,1}(w_1^ew_2^{q m \alpha}g(w_2^q)^{\alpha},w_2^q)\in \KK\nl w_1,w_2^q\nr.
\]
Moreover, because $B(\w, z) \in \KK\lb w_1,w_2^q\rb[z] $ and $B(w_1,\zeta w_2, \ovl\g_1( \zeta w_2)) \in \KK\nl w_1,w_2^q\nr $, we have that $B(\w, \ovl\g_1( \zeta w_2)) \in \KK\nl w_1,w_2^q\nr$ for every $q$-th root of  unity $\zeta$. We remark that $\ovl\G_1(w_2^q,z)$ may factor as a product of monic polynomials that are conjugated under the action of a subgroup $G$ of the $q$-th roots of unity. Thus,  the set $\{\ovl\g_1(\zeta w_2)\mid \zeta\in G\}$ contains exactly one root of every factor of $\ovl\G_1(w_2^q,z)$.
Therefore,  by definition \ref{temperate_fam} \ref{axiom:last} (and we highlight that this is the only point of the paper where Definition \ref{temperate_fam} \ref{axiom:last} intervenes), we conclude that:
\[
B(\w, \ovl\g_i(w_2)) \in \KK\nl w_1,w_2\nr
\]
for every $\ovl \g_i$ which is a root of $\ovl \Gamma_1$. Now, note that:
\[
B(\w, \ovl\g_i(w_2)) = w_2^{\beta m q } g(w_2^q)^{\beta} \cdot \psi_{i,1}(w_1^ew_2^{q m \alpha}g(w_2^q)^{\alpha},w_2^q)\in \KK\nl w_1,w_2\nr
\]
for every $i=1,\ldots,d/e'$. Since we also know that $\psi_{i,1} \in \KK \lb v_1^{1/e},v_2^{1/q} \rb$, we conclude from the fact that being temperate is closed under division, ramification and local blowings-up, see Proposition \ref{rk:Wtemp} \ref{rk:ramif}, \ref{rk:division} and Definition \ref{temperate_fam} \ref{axiom:monomialcomposition}, that $\psi_{i,1}(\vb) \in \KK\nl v_1^{1/e},v_2\nr$, finishing the proof.
\end{proof}

\subsection{Proof of Theorem \ref{thm:red2}}\label{ssec:red2}
Let $P$ be a Weierstrass polynomial in $y$ as in the statement of Theorem \ref{thm:red2}. Since $\Ker(\wdh\phi)$ is a prime ideal, $P$ is irreducible, so it is a reduced polynomial. In particular the discriminant of $P$ is a formal curve $\Delta(P)$. By resolution of singularities, there exists a sequence of point blowings-up
\[
\xymatrix{ (N_r,F_r) \ar[r]^{ \quad \,\, \tg_r}  &  \cdots   \ar[r]^{ \tg_2 \quad \, \,} &  (N_1,F_1)  \ar[r]^{\tg_1 \quad \quad}& (N_0,0)= (\wdh{\mathbb A}^2_\KK,0)}
\]
such that the discriminant of $P_{\pb} = \sigma^{\ast}_{\pb}(P)$ is everywhere monomial; we set $\s = \s_1 \circ \cdots \circ \s_r$. Apart from blowing-up the origin once, we can always suppose that the sequence of blowings-up has at least one blowing-up, that is, $r\geq 1$. In particular, the blowing-up $\s_1: (N_1,F_1) \lgw (N_0,0)$ is always defined. Now, there exists a point $\pb \in F_1$ and a temperate coordinate system $\vb = (v_1,v_2)$ where $(\s_1)_{\pb}^{\ast}$ is given by $(x_1,x_2) = (v_1,v_1v_2)$. It follows from the expressions of $\phi$ and $P$ given in the statement of Theorem \ref{thm:red2}, that:
\[
(\s_{\pb})_1^*(P)=P(v_1,v_1v_2,y), \quad \widetilde{\phi} (v_1,v_2) = (v_1,v_2,\psi(\vb)),
\]
are such that $\widetilde{\phi} \circ \s =\phi$; moreover, since $P\in \Ker(\phi)$, $(\s_{\pb})_1^*(P)$ is divisible by $y-\psi(v_1,v_2)\in \OOv[y]$ and, therefore, admits a temperate factor. We conclude that there exists a point $\pb_0 \in F_r$ where $P_{\pb_0} = \sigma^{\ast}_{\pb_0}(P)$ admits a temperate factor. In order to finish the proof, it is enough to prove the following result:

\begin{proposition}[{cf. \cite[Proposition 4.6]{BCR}}]\label{prop:inductivescheme}
Let $P \in \KK\lb\x\rb [y]$, and let $\s: (N_r,F_r) \lgw (N_0,0)$ be a sequence of point blowings-up such that the discriminant of $P\circ \s$ is everywhere monomial. Let $\pb\in F_r^{(k)}$ be such that $P_{\pb} = \s^*_\pb(P)$ has a temperate factor. Then $P$ has a non-constant temperate factor.
\end{proposition}
\begin{proof}

We prove this result by induction on the lexicographical order on $(r,k)$. First, suppose that $(r,k) = (r,1)$ with $r\geq 1$. By Theorem \ref{local_to_semiglobal}, there is a non-constant factor $Q\in\POx[y]$ of $P$. Without loss of generality, we may suppose that $Q$ is the monic factor of $P$ in $\POx[y]$ of maximal degree. Note that $Q$ extends at every point of $F_1^{(r)}$ by Theorem \ref{thm:extension}, and furthermore, $Q$ extends temperately by Lemma \ref{lemma:TempExt}. If $r=1$, then we conclude from Lemma \ref{lem:inductive} that $Q\in \KK\lb \x\rb [y]$, so that $Q\in \OOx[y]$ since $\KK\lb \x\rb  \cap \POx = \KK\nl \x\nr$ by Lemma \ref{lemma:TempExt}. 

If $r>1$, let $\pa_1,\ldots,\pa_j$ be the points of $F_1^{(1)}$ that are centres of subsequent blowings-up, and denote $P_i:=(\s_1)_{\pa_i}^*(P)$. Denote $\s':=\s_2\circ\cdots\circ\s_r$. By the induction hypothesis, for every $i$, $P_i$ has a temperate factor. Indeed, denoting by $\pb_i$ the point of $F_r^{(1)}$ which is sent to $\pa_i$ by $\s'$, we get that $\s^*_{\pb_i}(Q)$ is a temperate factor of ${\s'}^*_{\pb_i}(P_i)={\s}^*_{\pb_i}(P)$, obtained after only $r-1$ blowings-up. 

Now, denote by $q_i$ the monic temperate factor of $P_i$ of maximal degree. Then ${\s'}^*_{\pb_i}(q_i)$ is a temperate factor of $(\s')^*_{\pb_i}(P_i)$ such that $(\s')^*_{\pb_i}(P_i/q_i)$ has no non-constant temperate factor, otherwise by induction hypothesis, $P_i/q_i$ would have a non-constant temperate factor. Next, note that, by Theorem \ref{local_to_semiglobal}, ${\s}^*_{\pb_i}(Q)$ is also a temperate factor of ${\s}^*_{\pb_i}(P)$ such that ${\s}^*_{\pb_i}(P/Q)$ has no non-constant temperate factor. We conclude that ${\s'}^*_{\pb_i}(q_i)={\s}^*_{\pb_i}(Q)$, hence $q_i$ coincides with $(\s_1)_{\pa_i}^*(Q)$, which therefore admits a temperate extension at $\pa_i$. Moreover at every point $\pb'$ of $F_1^{(1)}\setminus \{\pa_1,\ldots,\pa_j\}$, $(\s_1)_{\pb'}^*(Q)$ is temperate, since it coincides with ${\s}_{\pb'}^*(Q)$. We conclude now, exactly as in the case $r=1$, that lemma \ref{lem:inductive} implies that $Q$ is a temperate factor of $P$.

Finally, suppose that $(r,k)$ is such that $k> 1$, and denote by $\pa\in F_{k-1}^{(j)}$ the center of $\s_k$, for some $j\leq k-1$. Denote $P_\pa := \left(\s_{1}\circ\cdots\circ\s_{k-1}\right)_\pa^*(P)$. Then by the induction hypothesis, $P_\pa$ has a non-constant temperate factor. Therefore, by denoting $ \s' :=\s_k\circ\cdots\circ\s_r$, at every point $\pb'\in(\s')^{-1}(\pa)$, the polynomial $ (\s')_{\pb'}^*(P_\pa)$ has a temperate factor. In particular, if $\pb'\in  (\s')^{-1}(\pa)\cap F_r^{(j)}$, we get that $\s^*_{\pb'}(P)$ has a temperate factor at a point of $F_r^{(j)}$ with $j<k$, and we conclude by induction.

\end{proof}


\section{Examples of W-temperate families}
\subsection{Algebraic power series}
When $\KK$ is a field, we denote by $\KK\langle x_1,\ldots, x_n\rangle$ the subring of $\KK\lb x_1,\ldots, x_n\rb$ of formal power series that are algebraic over $\KK[x_1,\ldots, x_n]$. We have the following proposition:

\begin{proposition} Let $\KK$ be an uncountable algebraically closed  field of characteristic zero.  The family of algebraic power series rings ($\KK\langle x_1,\ldots, x_n\rangle)_n$ is a minimal W-temperate family, that is, it is  contained in every other W-temperate family.
\end{proposition}

\begin{proof}
Let $(\KK\nl x_1,\ldots, x_n\nr)_n$ be a arbitrary W-temperate family. Since $\KK\nl \x \nr$ is a Henselian local ring containing $\KK[\x]_{(\x)}$, and since $\KK\langle \x\rangle$ is the Henselization of $\KK[\x]_{(\x)}$, we have
\(
\KK\langle \x \rangle\subset\KK\nl \x\nr
\)
by the universal property of the Henselization. 

Next, let us prove that $(\KK\langle x_1,\ldots, x_n\rangle)_n$ is a W-temperate family. The first four axioms of Definition \ref{W-fam} are classical, while the fifth axiom has been proved by Lafon in \cite{La}, see also \cite{Rond}. So, let us check that Definition \ref{temperate_fam} is verified. Once again, axiom \ref{axiom:monomialcomposition} is straightforward, and we consider:

\medskip
\noindent
\textbf{Axiom \ref{axiom:Abhyankar-Moh} of Definition \ref{temperate_fam}:} We prove the contrapositive of the axiom, that is, let $F\in\KK\lb \x \rb$ be such that 
\[
W:=\left\{\la\in \KK\mid  F(\x', \la x_1)\in \KK\langle \x'\rangle\right\}
\]
is uncountable and let us prove that $F \in \KK \langle \x'\rangle $. Let us denote by $\KK_0$ the algebraic closure of the field extension of $\Q$ generated by the coefficients of $F$. Since $F$ has a countable number of coefficients, $\KK_0$ is a countable field. Let $\la\in W\setminus \KK_0$; in particular $\la$ is transcendental over $\KK_0$. By assumption on $W$, we have
\begin{equation}\label{alg_equations}
a_0(\x')F(\x',\la x_1)^d+a_1(\x')F(\x',\la x_1)^{d-1}+\cdots+ a_d(\x')=0
\end{equation}
where the $a_i(\x')\in\KK[\x']$. Let us denote by $(\underline a)$ the vector whose entries are the coefficients of the $a_i(\x')$. Then \eqref{alg_equations} is satisfied if and only if $(\underline a)$ satisfies a (countable) system of linear equations $(\mathcal S)$ whose coefficients are in $\KK_0(\la)$ (determined by the vanishing of the coefficients of each monomial ${\x'}^\a$ for $\a\in\N^{n-1}$). And $(\mathcal S)$ is equivalent to a finite system of linear equations $(\mathcal S')$ with coefficients in $\KK_0(\la)$. And this system has a nonzero solution in $\KK$ if and only if it has a nonzero solution in $\KK_0(\la)$, and this solution yields non trivial polynomials $\wdt a_i(\x')\in\KK_0(\la)[\x']$ such that
$$
\wdt a_0(\x')F(\x',\la x_1)+\wdt a_1(\x')F(\x',\la x_1)^{d-1}+\cdots+ \wdt a_d(\x')=0.
$$
By multiplying by some polynomial in $\KK_0[\la]$ we may assume that the $\wdt a_i(\x')$ belong to $\KK_0[\x'][\la]$, thus we write $\wdt a_i=\wdt a_i(\x',\la)$. By dividing by a large enough power of $x_n-\la x_1$, we may assume that one of them is not divisible by $x_n-\la x_1$. Therefore not all the $\wdt a_i(\x',x_n/x_1)$ are zero, and
$$
a_0(\x',x_n/x_1)F(\x',x_n/x_1)+a_1(\x',x_n/x_1)F(\x',x_n/x_1)^{d-1}+\cdots+ a_d(\x',x_n/x_1)=0,
$$
whence $F(\x)\in\KK\langle \x\rangle$. This proves the result.
 
\medskip
\noindent
\textbf{Axiom \ref{axiom:last} of Definition \ref{temperate_fam}:} We follow the notation of axiom \ref{axiom:last}. For each $p_k(x_2,z)$, we consider its Euclidean division by the minimal polynomial $\G$ of $\g$:
\[
p_k(x_2,z) = \G(x_2,z) \cdot q_k(x_2,z) + r_k(x_2,z)
\]
where $\deg_z(r_k) <d = \deg_z(\G)$. By Lemma \ref{lemma:div-linear-bound}, there is $a\in\N$ such that $\deg_{x_2}(r_k(x_2,z))\leq a k$ for every $k$.

Note that $p_k(x_2,\g'(x_2)) = r_k(x_2,\g'(x_2))$ for every root $\g'(x_2)$ of $\G$, so we may consider the auxiliary function:
\[
Q(x_1,x_2,z) = \sum_{ k \in \mathbb{N}} x_1^k r_k(x_2,z) = \sum_{k=0}^{d-1} q_{k}(\x) z^k
\]
where $q_{i}(\x) \in \K[x_2] \lb x_1 \rb$, and note that $P(\x,\g'(x_2)) = Q(\x,\g'(x_2))$ for every root $\g'(x_2)$ of $\G$. Since  $\deg_{x_2}(r_k(x_2,z))\leq a k$ for every $k$, we may write $q_k(\x)=\wdh q_k(x_1,x_1x_2,\ldots, x_1x_2^a)$ for some formal power series $\wdh q_k(x_1,y_1,\ldots, y_ a)\in\KK\lb x_1,\y\rb$. Now, there exist  formal power series $\wdh g_i$, for $i=1$, \ldots, $a$,  and $\wdh k$ such that:
\[
\begin{aligned}
Q(\x,\g(t))=\sum_{k=0}^{d-1}\wdh q_k(x_1,\y)z^k + &\sum_{i=1}^a (y_i-x_1x_2^i)\wdh g_i(\x,\y,t,z)+ 
    (z- \g(t))\wdh k(\x,\y,t,z).
    \end{aligned}
\]
By the nested approximation Theorem for linear equations (see \cite[Theorem 3.1]{CPR}) this equation has a non trivial nested algebraic solution
 $$
 (\wdt q_k(x_1,\y), \wdt g_i(\x,\y,t,z)), \wdt k(\x,\y,t,z))\in \KK\langle x_1,\y\rangle^d\times\KK\langle \x,\y,t,z\rangle^{a+1}.
 $$
 In particular $\wdt Q:=\wdt q_0(x_1,x_1x_2,\ldots, x_1x_2^a)+\cdots+\wdt q_{d-1}(x_1,x_1x_2,\ldots, x_1x_2^a)z^{d-1}$ is an algebraic power series satisfying
 $$
 \wdt Q(\x,\g(x_2))=Q(\x,\g(x_2)).
 $$
Moreover we have $\wdt Q= \sum_{ k \in \mathbb{N}} x_1^k \wdt r_k(x_2,z) $ where the $\wdt r_k(x_2,z)\in\KK[x_2,z]$ with $\deg_z(\wdt r_k)\leq d-1$ and $\deg_{x_2}(\wdt r_k)\leq ak$. So, since for every $k$, $r_k(x_2,\g(x_2))=\wdt r_k(x_2,\g(x_2))$, we have $\wdt r_k=r_k$ since the degree of the minimal polynomial of $\g(x_2)$ over $\KK[x_2]$ is $d$. In particular we have
$$P(\x,\g'(x_2))=Q(\x,\g(x_2))=\wdt Q(\x,\g'(x_2))\in\KK\langle x_1,x_2\rangle$$
for every root $\g'(x_2)$ of $\G$. This ends the proof.
\end{proof}
 
 \begin{lemma}\label{lemma:div-linear-bound}
Let $\G(\x,z)\in\KK[\x,z]$ be a monic polynomial in $z$ of degree $e$. Let $p(\x,z)\in\KK[\x,z]$ with $\deg_z(p)\leq d$, where $d\geq e-1$. Consider the division of $p$ by $\G$: 
$$p(\x,z)=\G(\x,z)q(\x,z)+r(\x,z)$$
with $\deg_z(r)<e$. Then $\deg_\x(r)\leq \deg_\x(p)+(d-e+1)\deg_\x(\G).$
\end{lemma}
\begin{proof}
The proof is made by induction on $d\geq e-1$. If $d=e-1$, it is clear. Assume that the result is proved for polynomials of degree $d-1$ where $d\geq e$. We can write
$$
p(\x,z)=\G(\x,r)\times b_e(\x)+\wdt b(\x,z)
$$
where $p_e(\x)$ is the coefficient of $z^e$ in $p(\x,z)$, and $\deg_z(\wdt p)<\deg_z(b)$.  Therefore $\deg_\x(\wdt p)\leq \deg_\x(p)+\deg_\x(\G)$. Since $p$ and $\wdt p$ have the same remainder $r$ by the division  by $\G(\x,z)$, we apply the inductive assumption to see that
\[
\deg_\x(r)\leq \deg_\x(\wdt p)+(d-e)\deg_\x(\G)\leq \deg_\x(p)+(d-e+1)\deg_\x(\G),
\]
finishing the proof.
\end{proof}


\subsection{Convergent  power series}\label{ssec:ComplexAnalyticTemperate}

The goal of this Section is to prove Corollary \ref{thm:Gabrielov_convergent}. We start by recalling the definition of convergent power series:

\begin{definition}
Let $\KK$ be a field equipped with an absolute value $|\cdot|$ for which $\KK$ is  complete. Let $f=\sum_{\a\in\N^n}f_\a\x^\a\in\KK\lb \x\rb$ be a power series. We say that $f$ is a \emph{convergent power series} if the series $f(\xi)$ converges for $\xi\in\KK^n$ small enough or, equivalently, if there exist real positive numbers $A$, $B>0$ such that
$$\forall\a\in\N^n,\ |f_\a|A^{|\a|}\leq B.$$
The set of convergent power series over $\KK$ is a subring of $\KK\lb x_1,\ldots, x_n\rb$ denoted by $\KK\{x_1,\ldots, x_n\}$.
\end{definition}

A \emph{morphism of analytic $\KK$-algebras} is a morphism of local rings $\phi:A\lgw B$, where $A=\fract{\KK\{\x\}}{I}$ and $B=\fract{\KK\{\ub\}}{J}$, and $\phi$ is induced by a morphism of power series rings $\KK\{\x\}\lgw \KK\{\ub\}$. For a morphism $\phi:A\lgw B$ of analytic $\KK$-algebras, we denote by $\Ar(\phi)$ the analytic  rank that is equal to $\dim\left(\fract{A}{\Ker(\phi)}\right)$.

\subsubsection{Convergent power series rings as W-temperate families}
 We begin by recalling the following result:

\begin{theorem}\label{thm:AM}\cite{AM}
Let $\KK$ be a field equipped with an absolute value such that that $\{|y|\mid y\in\KK\}$ is dense in $\R_+$, and let $f\in\KK\lb x_1,\ldots, x_n\rb$ be a divergent power series. Let $W$ be the set of $\la\in\KK$ such that $f(\la x_2,x_2,\ldots, x_n)$ is convergent. 
Then $W$ is a countable union of closed nowhere dense subsets of $\KK$.
\end{theorem}

\begin{proposition}\label{prop:cvgt-W}
Let  $\KK$ be a characteristic zero field that is equipped with a non trivial absolute value that makes $\KK$ a complete field. Then the family $(\KK\{x_1,\ldots, x_n\})_n$ is a Weierstrass family over $\KK$. Moreover, if we assume that $\KK$ is algebraically closed, then $(\KK\{x_1,\ldots, x_n\})_n$ is a W-temperate family over $\KK$.
\end{proposition}
\begin{proof}
Axioms \ref{axiom:inclusions} \ref{axiom:intersection_n+m} \ref{axiom:permutation} \ref{axiom:unit} of Definition \ref{W-fam}, and Axiom \ref{axiom:monomialcomposition} of Definition \ref{temperate_fam} are easily verified. Axiom \ref{axiom:weierstrass} of Definition \ref{W-fam} is well known (see \cite{Bo} for example). Therefore  $(\KK\{x_1,\ldots, x_n\})_n$ is a Weierstrass family over $\KK$. We assume, from now on, that $\mathcal{K}$ is algebraically closed.

\medskip
\noindent
First let us prove \ref{axiom:Abhyankar-Moh} of Definition \ref{temperate_fam}. First, we remark that, since $\KK$ is complete and the absolute value is not trivial, $\KK$ is uncountable. Moreover because $|\cdot|$ is non trivial, there is $y\in\KK$ such that $|y|\neq 0$ or $1$. Because $\KK$ is algebraically closed, $|\KK|$ contains the set $\{|y|^{p/q}\mid p, q\in\N^*\}$ that is dense in $\R_+$. Therefore we can apply Theorem \ref{thm:AM} to see that W is a countable union of closed nowhere dense subsets of $\KK$. But since $\KK$ is a complete metric space, by Baire category theorem, $\KK\setminus W$ is uncountable.

\medskip
\noindent
Now let us prove Axiom \ref{axiom:last}  of Definition \ref{temperate_fam}: Denote by $\Gamma \in \KK[t,z]$ the minimal polynomial of $\g$ and let $\g' \in \KK \lb t \rb$ be a conjugate root; note that $\g$ and $\g' \in \KK \{ t \}$ since they are algebraic. Consider the curve $\mathcal{C} \subset \KK^2_{t,z}$ given by $\Gamma[t,z]=0$, and let $\pi: \KK^2_{t,z} \lgw\KK_{t}$ be the projection $\pi(t,z)=t$. Since $\g$ and $\g' \in \KK \{ t \}$, there exists a closed ball $D_0 \subset \K_t$ centered at the origin such that, $\g$ and $\g'$ induce analytic functions on $ D_0$. If $\mathcal{K} = \mathbb{C}$, let $\mathcal{C}^{an}=\mathcal{C}$, and if $\mathcal{K}$ is a non-archimedean field, let $\mathcal{C}^{an}$ be the analytification of $\mathcal{C}$ in the sense of Berkovich. Denote by $D $ and $D' \subset \mathcal{C}^{an}$ the images of $\g$ and $\g'$ over $D_0^{an}$ respectively. Set $\xi=\g(0)$ and $\xi'=\g(0)$. The points $\xi$ and $\xi'$ are type-1 points of $\mathcal C^{an}$, therefore they admit a fondamental system of open neighborhoods in $\mathcal C^{an}$ that are open balls (see \cite[Corollary 4.27]{BPR}). Therefore, because $D_0^{an}$ is compact,  we may choose $D_0$ small enough such that there exist closed balls $B_1$ and $B_2$ with
$\xi\in B_1\subset D$ and $D'\subset B_2$. Now, recall that $P(\x,\g(x_2))\in \KK \{ x_1,x_2\}$, where:
\[
P(\x,z) = \sum_{k \in \mathbb{N}} x_1^k p_k(x_2,z)
\]
and $p_k(x_2,z)$ are polynomials such that $\deg_{x_2}(p_k) \leq \alpha k$ for some $\alpha \in \mathbb{N}$. In particular this implies that there exists $A$ and $B>0$ such that:
\[
\|p_k(x_2,\g(x_2))\|_{D_0 } \leq A \, B^k \, k!,\quad \forall k\in \mathbb{N}.
\]
But $D_0$ is dense in $D_0^{an}$, therefore we have
$$\|p_k(x_2,z)\|_{B_1}\leq \|p_k(x_2,z)\|_{D}=\|p_k(x_2,\g(x_2))\|_{D_0^{an} }\leq A \, B^k \, k!,\quad \forall k\in \mathbb{N}.$$
We claim that, up to replacing $D_0$ by a smaller closed ball, we have 
$$
\|p_k(x_2,z)\|_{B_2} \leq M^{\alpha k}\|p_k(x_2,z)\|_{B_1},
$$
where $M >0$. So we conclude that
\[
\|p_k(x_2,\g'(x_2))\|_{D_0} \leq \|p_k(x_2,z)\|_{B_2} \leq M^{\alpha k}\|p_k(x_2,z)\|_{B_1}  \leq A \, (M^{\alpha}B)^k \, k!,\quad \forall k\in \mathbb{N}
\]
which implies that $P(\x,\g'(x_2))\in \KK \{ x_1,x_2\}$ as we wanted to prove.

\bigskip
\noindent
\emph{Proof of the Claim:} 
Since $\mathcal{C}$ is an affine algebraic curve, it admits a compactification ${\mathcal{C}}'$ in the projective   space $\mathbb{P}^n_\KK$. Denote by $\pi: \wdt{\mathcal{C}} \to \mathcal{C}'$ the normalization, and note that $\wdt{\mathcal{C}}$ is a projective smooth curve. We still denote by $\pi$ the analytic map induced by $\pi$ between the analytification of $\wdt{\mathcal C}$ and $\mathcal C'$. We denote by $\wdt B_j = \pi^{-1}(B_j)$ for $j=1,2$. Consider the pullbacks $f_i:=\pi^*(z_i)$ of the coordinate functions $z_i$ of the affine space $\KK^n$. We denote by $q_1$, \ldots, $q_s \in \wdt{\mathcal{C}}$ the poles of the $f_i$ and we denote by $m_j$ the maximal multiplicity of the poles $q_j$ for the functions $f_1$, \ldots, $f_n$. Now, consider the divisor:
\[
D = m \,q_0 - \sum_{j=1}^s m_j q_j
\]
where $q_0\in\wdt{\mathcal C}$ an interior point of $\wdt{B}_1$ and $m = \sum_{j=1}^s m_j + \mbox{genus}(\wdt{\mathcal{C}})$. By the Riemann-Roch Theorem, there exists a rational function $h$ defined over $\wdt{\mathcal{C}}$ such that $(h) + D$ is an effective divisor, where $(h)$ denotes the divisor associated to $h$. Apart from enlarging $B_2$, we may suppose that $h$ is never zero over $\partial \wdt{B}_2$. If $\KK=\C$, it is always possible to enlarge $B_2$. If $\KK$ is non-archimedean, we may replace $D_0$  and $B_2$ by smaller ball, in such a way that we can enlarge $B_2$.

Let $d$ be the degree of $p_k$. Then $(p_k\circ\pi) \cdot h^d$ induces an analytic function on $\wdt{\mathcal{C}}^{an}\setminus\{q_0\}$. Therefore, since $h$ is non-zero over $\partial \wdt{B}_2$,  Lemma \ref{lem:max_principle} (or the Maximum principle when $\KK=\C$) given below implies that:
\[
\|p_k\|_{B_2} = \|p_k \circ \pi\|_{\wdt{B}_2} = \|p_k \circ \pi\|_{\partial \wdt{B}_2} \leq  \left(\| h^{-1} \|_{\partial \wdt{B}_2}\right)^d \cdot \|h^d \cdot p_k \circ \pi\|_{\partial \wdt{B}_2}
\]
Next, since $(p_k\circ\pi) \cdot h^d$ has no poles in $\ovl C^{an}\setminus \wdt{B}_1$, we conclude via Lemma \ref{lem:max_principle}  that:
\[
\|h^d \cdot p_k \circ \pi\|_{\partial \wdt{B}_2} \leq \|h^d \cdot p_k \circ \pi\|_{\ovl{\mathcal{C}}^{an} \setminus \wdt{B}_1} \leq \|h^d \cdot p_k \circ \pi\|_{ \partial \wdt{B}_1} \leq \left(\|h\|_{ \partial \wdt{B}_1}\right)^d \cdot \|p_k \circ \pi\|_{ \partial \wdt{B}_1}
\]
Finally, by combining the above inequalities, we get:
\[
\begin{aligned}
\|p_k\|_{B_2} 
&\leq \left( \| h^{-1} \|_{\partial \wdt{B}_2} \cdot \|h\|_{ \partial \wdt{B}_1} \right)^d\cdot \|p_k \circ \pi\|_{ \partial \wdt{B}_1}\\ &\leq  \left( \| h^{-1} \|_{\partial \wdt{B}_2} \cdot \|h\|_{ \partial \wdt{B}_1} \right)^d\cdot \|p_k\|_{B_1}.
\end{aligned}
\]
So the Claim is proved with $M = \| h^{-1} \|_{\partial \wdt{B}_2} \cdot \|h\|_{ \partial \wdt{B}_1}$. 
\end{proof}

\begin{remark}
The proof of the previous proposition (especially the fact that convergent power series satisfy Axiom  \ref{axiom:last}  of Definition \ref{temperate_fam})
works over any algebraically closed complete valued field $\KK$ of characteristic zero. But when $\KK$ is archimedean, that is $\KK=\C$, a different proof is given by \cite[Lemma 5.36]{BCR}.
\end{remark}

We warmly thank Charles Favre for providing us the sketch of the proof of the following Lemma in a private communication.

\begin{lemma}\label{lem:max_principle}
Let $\KK$ be a non-archimedean complete valued field of characteristic zero that is algebraically closed.
Let $U$ be a connected open subset of a smooth projective analytic curve $\mathcal C$. Let  $f$ be an analytic function on $U$. If $|f|_{|_U}$ has a maximum at a point of $U$ then $f$ is constant on $U$. Moreover if $f$ extends continuously on $\ovl U$,  then $|f|$ attains its maximum on $\ovl U$ at a point of  $\partial U=\ovl U\setminus U$.
\end{lemma}

\begin{proof}
 Let $x$ be a point of $U$ at which $|f|$ has a local maximum, and assume that $f$ is not constant. If $x$ is of type 1, 3 or 4, then $x$ has a neighborhood $D$ included in $U$ that is an open ball or an open annulus  (see \cite[Corollary 4.27]{BPR}). Therefore we may assume that $D$ is in the Berkovich line, and we get a contradiction by the Maximum Principle for domains in the Berkovich line (cf. \cite[Proposition 3.4.5]{BR}). If $x$ is of type 2, we set $F:=-\log|f|$. Then, by \cite[Theorem 5.15 (4)]{BPR}, the sum of the outgoing slopes of $F$ at $x$ equals 0. But since $x$ is a local maximum of $F$, then all the outgoing slopes of $F$ are zero. Let $\Sigma$ be the skeleton of $U$ with respect to some semistable vertex set of $U$ containing $x$ (see \cite[Definitions 3.1 \& 3.3]{BPR}). By  \cite[Theorem 5.15 (2)]{BPR}, $F$ is piecewise linear on $\Sigma$. Therefore, $F_{|_{\Sigma_x}}$ is constant, where $\Sigma_x$ denotes the connected component of $\Sigma$ containing $x$. Therefore  $F$ is locally constant at $x$ because, by \cite[Theorem 5.15 (1)]{BPR}, $F=F\circ \tau$ where $\tau:U\lgw \Sigma$ is the retraction of $U$ onto its skeleton $\Sigma$. This is a contradiction, thus $f$ is constant. Since $\ovl U$ is compact, $|f|$ has a maximum on it. If the maximum is attained at a point of $U$, then $|f|$ is constant on $U$, thus on $\ovl U$ by continuity. Thus in any case the maximum is attained at a point of $\partial U$. 
\end{proof}

\subsubsection{Proof of Corollary \ref{thm:Gabrielov_convergent}}

We just need to prove the following result (the reduction of Corollary \ref{thm:Gabrielov_convergent} to the following result follows the lines of Section \ref{red:sec}, see Remark \ref{rk:red1}, because the family of convergent power series over a characteristic zero complete field is a Weierstrass family by Proposition \ref{prop:cvgt-W}):
 \begin{theorem}\label{thm:red1_Gab}
Let $\phi:\KK\{ x_1,\ldots,x_n, y\}\lgw \KK\{ u_1,\ldots, u_n\}$ be a morphism of convergent power series rings over a characteristic zero complete field $\KK$ such that
\begin{enumerate}
\item[i)]  The kernel of $\wdh\phi$ is generated by one Weierstrass polynomial $P\in\KK\lb \x\rb[y]$.
\item[ii)] $\Gr(\phi)=\Fr(\phi)=n$.
\end{enumerate}
Then  $P\in\KK\{\x\}[y]$.
\end{theorem}

\begin{proof} We denote by $\ovl\KK$ the completion of an algebraic closure of $\KK$. This is an algebraically closed field.
Therefore we can apply Theorem \ref{thm:TemperateRank} to the W-temperate family $(\ovl\KK\{x_1,\ldots, x_n\})_n$. This shows  that $P\in\ovl\KK\{\x\}[y]$. Therefore $P\in\ovl\KK\{\x\}[y]\cap\KK\lb\x\rb[y]=\KK\{\x\}[y]$.
\end{proof}




\subsection{Eisenstein power series}\label{ssec:Eisenstein}


Let $\O$ be a UFD, and let $\KK$ be an algebraic closure of its fraction field. The ring of Eisenstein series over $\O$ is the filtered limit of rings:
$$\bigcup_{\mathfrak c\in\KK}\bigcup_{f\in \O\setminus\{0\}} \O_f\lb x_1,\ldots, x_n\rb[\mathfrak c]$$
where $\O_f$ denotes the localization of $\O$ with respect to the multiplicative family $\{1,f,f^2,\ldots,\}$. 

The main result of this subsection is the following:

\begin{proposition}\label{prop:EisensteinIsTemperate}
If $\O$ is a  UFD containing an uncountable characteristic zero field $\k$, the ring of Eisenstein series is a W-temperate family over $\KK$.
\end{proposition}

\begin{proof}
Axioms \ref{axiom:inclusions} \ref{axiom:intersection_n+m} \ref{axiom:permutation} \ref{axiom:unit} of Definition \ref{W-fam} are easily verified. 

\medskip
\noindent
\textbf{Axiom \ref{axiom:weierstrass} of Definition \ref{W-fam}:} consider $F$ and $G\in\O_f\lb\x\rb[\mathfrak c]$ as in the statement of Axiom \ref{axiom:weierstrass}. We have $F(0,x_n)=x_n^du(x_n)$. If we multiply $f$ by $u(0)$, we may assume that $u(x_n)$ is a unit in $\O_f[\pc]\lb x_n\rb$. Let $\LL$ be the fraction field of $\O_f[\pc]$. By the Weierstrass division theorem for power series in $\LL\lb \x\rb$, $G=QF+R$ where $Q\in\LL\lb \x\rb$ and $R\in\LL\lb \x'\rb[x_n]$ and $\deg_{x_n}(R)<d$. We claim that the coefficients of $Q$ and $R$ are also in $\O_f[\pc]$. Indeed, fix the following order on the monomials: We have
$\x^\a<\x^\b$  if  
$$
(\a_1+\cdots+\a_{n-1}+(d+1)\a_n, \a_1,\ldots, \a_n)<_{\text{lex}}(\b_1+\cdots+\b_{n-1}+(d+1)\b_n, \b_1,\ldots, \b_n)
$$
where $<_{\text{lex}}$ denotes the lexicographic order. In particular the nonzero monomial of least order in the expansion of $F$ is $Cx_n^d$ where $C$ is a unit in $\O_f[\pc]$. For a series $H\in\LL\lb\x\rb$ we denote by $\ini(H)$ the monomial of least weight in the expansion of $H$. 

We now consider an inductive way to construct the unique coefficients $Q$ and $R$. We start by setting $G^{(0)}=G$, $Q^{(0)}=0$ and $R^{(0)}=0$. Fix $k\geq 0$, and assume that  $Q^{(\ell)}$ and $R^{(\ell)}$ have been constructed for every $\ell\leq k$  in such a way that
$G^{(\ell)}=G-FQ^{(\ell)}-R^{(\ell)}$ satisfies $\ord(G_{\ell+1})\geq \ord(G_{\ell})$. We consider the two following cases:
\begin{enumerate}
\item[i)] If $\ini(G^{(k)})$ is divisible by $x_n^d$, we set $R^{(k+1)}:=R^{(k)}$ and $Q^{(k+1)}:=Q^{(k)}+\fract{\ini(G^{(k)})}{\ini(F)}$.
\item[ii)] If $\ini(G^{(k)})$ is not divisible by $x_n^d$, we set $R^{(k+1)}:=R^{(k)}+\ini(G^{(k)})$ and $Q^{(k+1)}:=Q^{(k)}$.
\end{enumerate}
By the formal Weierstrass division Theorem, this process converges as $G^{(k)}\lgw 0$, $Q^{(k)}\lgw Q$ and $R^{(k)}\lgw R$ when $k\lgw\infty$. But we see that we do not need to introduce elements of $\LL$ that does not belong to $\O_f[\pc]$ because the coefficient of initial term of $F$ is a unit in $\O_f[\pc]$. This proves the claim.

\medskip
\noindent
\textbf{Axiom \ref{axiom:monomialcomposition} of Definition \ref{temperate_fam}:} This property is easily verified.

\medskip
\noindent
\textbf{Axiom \ref{axiom:Abhyankar-Moh} of Definition \ref{temperate_fam}:} The property follows from the following Lemma, which is stronger version of the axiom valid for Eisenstein power series:

\begin{lemma}\label{lem:reduction}
In the conditions of the statement of Proposition \ref{prop:EisensteinIsTemperate}, consider $F\in\mathbb \KK\lb \x \rb$ and assume that $F\notin \KK\nl \x\nr$. Then the following set is countable
\[
W:=\left\{\la\in \k\mid  F(\x', \lambda x_1)\in \KK\nl \x'\nr\right\}.
\]
\end{lemma}
\begin{proof}
We start by a general claim. Let $P(x_1,x_2)\in \O[x_1,x_2]$ be a homogeneous polynomial. Write
\[
P=\sum_{k=0}^dp_{k}x_1^{d-k}x_2^k,
\]
so that $P(x_1, \la x_1)=\left(\sum_k p_k\la^k\right)x_1^d$. Let $g\in \O$, $g\neq 0$; we claim that if  $\gcd(p_k, k=0,\ldots, d)=1$, then $\gcd\left(\sum_k p_k\la^k, g\right)\neq 1$ for at most finitely many $\la\in \k$. Indeed, assume that $\gcd\left(\sum_k p_k\la^k, g\right)\neq 1$ for infinitely many $\la\in \k$. Since $g$ has finitely many factors, this implies that $g$ has an irreducible factor $h$ such that, for infinitely many $\la\in \k$,
$h$ divides $\sum_k p_k\la^k$. Hence the polynomial $Q(T):=\sum_kp_kT^k\in \Frac(\fract{\O}{(h)})[T]$ has infinitely many roots in $\k$, which is  possible only if $h$ divides all the $p_k$ since $\Frac(\fract{\O}{(h)})$ is an infinite field (it is a  field containing $\k$). This contradicts the hypothesis, proving the Claim.

Now, we prove the contrapositive of the Lemma, that is, consider an element $F\in\mathbb \KK\lb \x\rb$ such that 
\[
W:=\left\{\la\in \k\mid  F(\x',\la x_1)\in \KK\nl \x'\nr\right\}
\]
is uncountable, and let us prove that $F\in\mathbb \KK\nl \x\nr$. Let $\LL$ be the fraction field of $\O$. Since $F$ has countably many coefficients, the field extension of $\LL$ generated by the coefficients of $F$ is a $\LL$-vector space of countable dimension. Let $(\pc_k)_{k\in\N}$ be a $\LL$-basis of this vector space, so
$$
F(\x)=\sum_{k\in\N}\pc_kF_k(\x)
$$ 
where the $F_k(\x)$ are in $\LL\lb \x\rb$ and, for each $\a\in\N^n$, the coefficient of $\x^\a$ is zero in all but finitely many $F_k(\x)$. Moreover, we can write
$$
F_k(\x)=\sum_{\a \in\N^{n-2}}\left(\sum_{d\in\N}\frac{P_{k,\a,d}(x_1,x_n)}{g_{k,a,d}}\right) x_2^{a_2}\cdots x_{n-1}^{\a_{n-2}}
$$
where the $P_{k,\a,d}\in \O[x_1,x_n]$ are homogeneous polynomials of degree $d$, and $g_{k,\a,d}$ is coprime with the gcd of the coefficients of $P_{k,\a,d}$. Now $\gcd(P_{k,\a,d}(1,\la), g_{k,\a,d})=1$ for $\la\in E_{k,\a,d}$ where $E_{k,\a,d}\subset W$ is cofinite by the Claim. Thus the complement of the set $E:=\cap_{k,\a,d} E_{k,\a,d}$ in $\k$ is at most countable. Therefore $E\cap W\neq 0$ if $\k$ is uncountable. Hence, by choosing $\la\in E\cap W$, there is $f\in \O$ such that $F_k(\x',\la x_1)\in \O_f\lb \x'\rb$ for every $k$. Then we see that for every $k$, $\a$ and $d$, $g_{k,\a,d}$ divides a power of $f$, whence $F(\x)\in \O_f\lb \x\rb$.
\end{proof}

\medskip
\noindent
\textbf{Axiom \ref{axiom:last} of Definition \ref{temperate_fam}:} The proof of this result is based on the following Galois-type result whose proof we postpone to $\S\S\S$\ref{sssec:Galois}

\begin{theorem}\label{thm_galois}
Let $A$ be a UFD, $\mathbb L$ be its fraction field and $\mathfrak{c}$ be in an algebraic closure of $\mathbb L$ and be separable over $\LL$. Let $\Gamma$ in $A[t,z]$ be irreducible. Assume that $\Gamma$ splits as a polynomial with coefficients in $A[\mathfrak c]\lb t\rb$ and let the $\g_i(t)\in A[\mathfrak c]\lb t\rb$ denote the roots of $\Gamma$. Then there is $f\in A$ such that, for every $Q\in\mathbb L[t,z]$:
 $$
 Q(t,\g_1)\in A[\mathfrak c]\lb t\rb\Longrightarrow Q(t,\g_2)\in A_f[\mathfrak c]\lb t\rb.
 $$
\end{theorem}

We now follow the notation of axiom \ref{axiom:last}. By the definition of Eisenstein power series and the assumption, we know that $P(\x,\g(x_2))\in\O_g[\pc]\lb x_1,x_2\rb$ for some $g\in\O$ and some $\pc \in \KK$; note that $\pc$ is separable over $\KK$ since $\KK$ is of characteristic zero. Moreover, by assumption:
\[
P(\x,z) = \sum_{k \in \mathbb{N}} x_1^k p_k(x_2,z)
\]
where $p_k(x_2,z)$ are polynomials such that $\deg_{x_2}$ is bounded by a liner function in $k$. Since $P(\x,\g(x_2))\in\O_g[\pc]\lb x_1,x_2\rb$, we conclude that $p_k(x_2,\g(x_2)) \in  \O_g[\pc]\lb x_2\rb$. Let $\g'$ be a conjugate root of $\g$. By Theorem \ref{thm_galois} applied to $A=\O_g$, there is $f\in\O$ such that, for every $k\in\N$, $p_k(x_2,\g'(x_2))\in O_{fg} [\pc]\lb x_2\rb$. Thus $P (\x,\g'(x_2))\in\O_{fg}[\pc]\lb \x \rb$, proving that the axiom is verified.
\end{proof}

\begin{remark}
The following example shows that we really need $f$ in the statement of Theorem \ref{thm_galois}. Let $f\in A$ irreducible and let $\Gamma(t,z)=z^2-(1+ft)$. So $\g_1=\sqrt{1+ft}$ and $\g_2=-\sqrt{1+ft}$. For $Q=\frac1f(1-z)$, we have
$$
Q(\g_1)=\frac1f\left(1-\sqrt{1+ft}\right)\in A\lb t\rb,\ \ 
Q(\g_2)=\frac1f\left(1+\sqrt{1+ft}\right)\in A_f\lb t\rb\setminus A\lb t\rb.
$$
\end{remark}


\subsubsection{Algebraic power series with coefficients in a UFD and proof of Theorem \ref{thm_galois}}\label{sssec:Galois}

In this subsubsection, we provide a proof of Theorem \ref{thm_galois}, and we collect results concerning algebraic power series which are of independent interest. We start with a simple Lemma:

\begin{lemma}\label{bound_degree}
Let $\G(t,z)\in A[t,z]$ be a polynomial with coefficients in an integral domain $A$. Let us write $\G=a_0(t)z^d+\cdots+a_d(t)$. Assume that $\G(t,z)$ has a root in $A[t]$ of degree $D$. Then $D\leq \max_i\{\deg_t(a_i(t))\}$.
\end{lemma}

\begin{proof}
After changing the indices we may assume that $a_0(t)\neq 0$. Let $F(t)\in A[t]$ with $\deg_t(F(t))=D$ and assume that $D>\max_i\{\deg_t(a_i(t))\}$. Then for $i>0$:
$$\deg_t(a_i(t)F(t)^{d-i})\leq \max_j\{\deg_t(a_j(t))\}+(d-i)D<dD\leq\deg_t(a_0(t)F(t)^d).$$
Therefore $\G(t,F(t))\neq 0$.
\end{proof}

The next Lemma shows that the coefficients of an algebraic infinite series over an UFD satisfies strong relations:


\begin{lemma}\label{finite_reduction}
Let $A$ be a UFD, $\mathfrak c$ in an algebraic closure of $\Frac(A)$ be finite and separable  over $A$. Let $\mathcal P$ be a representative family of primes of $A$ (i.e. each principal prime ideal of $A$ is generated by a unique element of $\mathcal{P}$) and $F(t)\in A[\mathfrak c]\lb t\rb\setminus A[\mathfrak c,t]$ be algebraic over $A[t]$. Then the following set is finite
$$
\left\{g\in \mathcal P\mid F(t)\in A[\mathfrak c,t]+(g)A[\mathfrak c]\lb t\rb\right\}.
$$
\end{lemma}

\begin{proof}
We start by showing that we can reduce the Lemma to the case that $F(t)\in A\lb t\rb\setminus A[t]$, that is, $F$ is independent of $\pc$. Indeed let $e$ denote the degree of $\mathfrak c$ over $A$. Then $F(t)$ can be written in a unique way as
$$
F(t)=F_0(t)+F_1(t)\mathfrak c+\cdots+ F_{e-1}(t)\mathfrak c^{e-1}
$$
where the $F_i(t)$ belong to $A\lb t\rb$. We denote by $\mathfrak c_2$, \ldots, $\mathfrak c_e$ the distinct conjugates of $\mathfrak c$ over $A$.  The power series $F^{(j)}(t)=F_0(t)+F_1(t)\mathfrak c_j+\cdots+ F_{e-1}(t)\mathfrak c_j^{e-1}$ for $j=2$, \ldots, $e$ are the conjugates of $F(t)$ over $A\lb t\rb$, therefore they are also algebraic over $A[t]$. We have
   $$
   \begin{bmatrix} F(t)\\ F^{(2)}(t)\\ \vdots \\ F^{(e)}(t) \end{bmatrix}=\begin{bmatrix}1 & \mathfrak c &\mathfrak c^2 & \cdots &\mathfrak c^{e-1}\\
 1 & \mathfrak c_2 &\mathfrak c_2^2 & \cdots& \mathfrak c_2^{e-1}\\
 \vdots & \vdots & \vdots & \ddots & \vdots\\
 1 & \mathfrak c_e &\mathfrak c^2_e & \cdots &\mathfrak c^{e-1}_e
   \end{bmatrix}\cdot
   \begin{bmatrix} F_0(t)\\F_1(t)\\ \vdots\\ F_{e-1}(t)\end{bmatrix}$$
The Vandermonde matrix is invertible, its entries are algebraic over $A$, thus the entries of its inverse are algebraic over $A$. Therefore the $F_i(t)\in A\lb t\rb$ are algebraic over $A[t]$. Thus, it is enough to prove the lemma for the $F_i(t)$; we may therefore assume that $F(t)\in A\lb t\rb$.
   
Write $F(t)=\sum_{k\in\N} F_kt^k$ with $F_k\in A$ for every $k$. Let $g\in A$ and $N\in \mathbb{N}^{\ast}$. We have that $F(t)$ is equal to a polynomial of degree  $\leq N$ in $t$ modulo $gA\lb t\rb$ if and only if 
$$
\forall k>N,\ \ F_k\in (g).
$$
Therefore for distinct $g_1$, \ldots, $g_s\in\mathcal P$ and because $A$ is a UFD, $F(t)$ is equal to a polynomial of degree $\leq N$ in $t$ modulo every $g_iA\lb t\rb$ if and only if
$$
\forall k>N,\ \ F_k\in (g_1\cdots g_s).
$$
Since $F(t)\notin A[t]$, we conclude that there does not exist an infinite subset $\mathcal G_N \subset \mathcal P$ such that for every $g\in \mathcal G_N$, $F(t)$ is equal to a polynomial  of degree $\leq N$ in $t$ modulo $gA\lb t\rb$. In particular, if we assume by contradiction that the set $\left\{g\in \mathcal P\mid F(t)\in A[t]+(g)A\lb t\rb\right\}$ is not finite, then there exists a sequence $(g_n)$ of distinct primes in $\mathcal P$, such that $F(t)$ is equal to a polynomial of degree $N_n$ modulo $g_nA\lb t\rb$ where the sequence $(N_n)_n$ is increasing and tends to infinity. In what follows, we show that the existence of this sequence would contradict Lemma \ref{bound_degree}.

Indeed, since $F(t)$ is algebraic, we may consider $\G(t,z):=a_0(t)z^d+\cdots+a_d(t)\in A[t,z]$ a polynomial such that $\G(t,F(t))=0$ and $a_0(t)\neq 0$. Denote by $F_n(t)$ (resp. $\G_n(t,z)$) the image of $F(t)$ in $A/(g_n)\lb t\rb$ (resp. of $\G(t,z)$ in $A/(g_n)[t,z]$). We have $\deg_t(F_n(t))=N_n$ and $\G_n(t,F_n(t))=0$. For $n\in\N$ large enough we have that $\G_n(t,z)\neq 0$ and $\deg_z(\G_n(t,z))=d$, because any given $a\in A$ has finitely many prime divisors. We conclude from Lemma \ref{bound_degree} that $N_n \leq \max_i\{\deg_t(a_i(t))\}$ for every $n$ sufficiently big, yielding a contradiction.
\end{proof}

\begin{remark}
Recall that, in general, an irreducible polynomial $\G(z)$ with coefficients in a UFD may be reducible modulo infinitely many primes of $A$. One classical example is given by $\G(z)=z^4+1$ that is irreducible over $\Z[z]$ but reducible modulo every prime number $p$. In contrast, Lemma \ref{finite_reduction} guarantees that for an irreducible polynomial $\G(t,z)\in A[t,z]$, the set 
$$
\left\{g\in \mathcal P\mid \G(t,z) \text{ is reducible modulo }(g)\right\}
$$
is finite, provided that $\G$ has a root in $A[\mathfrak c]\lb t\rb\setminus A[\mathfrak c][ t]$.
\end{remark}


Before proving Theorem \ref{thm_galois}, recall that given a UFD $A$ and $f\in A$, $f\neq 0$, the the localization $A_f$ is also a UFD; we will use this observation implicitly below. We recall that this claim follows from the fact that a UFD is a Krull domain in which every prime ideal of height 1 is principal. Since $A$ is a UFD, it is a Krull domain so $A_f$ is also a Krull domain. Because the localization morphism $A\lgw A_f$ induces an isomorphism between the primes of $A_f$ and the primes of $A$ avoiding $f$, every prime ideal of $A_f$ of height 1 is necessarily principal. We are now ready to prove our main result about algebraic power series with coefficients in a UFD:

\begin{proof}[Proof of Theorem \ref{thm_galois}]
We start by showing that we may suppose that $\Gamma$ is monic. We write
$$
\Gamma=p_0z^d+p_1z^{d-1}+\cdots+ p_d.
$$
We have
$p_0^{d-1}\Gamma(t,z)=R(t,p_0z)$ where 
$$
R(t,z)=T^d+p_1T^{d-1}+p_2p_0T^{d-2}+\cdots+ p_dp_0^{d-1}.
$$
We have $R(t,p_0\g_i)=0$ for $i=1$, $2$. If we set $\g'_i=p_0\g_i$, and we prove the statement of the Theorem for the $\g_i'$ then we also deduce the statement of the Theorem for the $\g_i$, since $Q(t,p_0^{-1} z) \in \LL[t,z]$ if and only if $Q\in \LL[t,z]$; therefore, we suppose that $\Gamma$ is monic in $z$.

Let us first treat the case that $\mathfrak c\in \mathbb L$. By replacing $A$ by $A_g$ for some well chosen $g\in A$, we can assume that $\mathfrak c\in A$. We claim that there exists $f\in A$ such that
\begin{equation}\label{implication}
\forall P\in A[t,z],\forall g\in A,\ P(t,\g_1)\in gA\lb t\rb\Longrightarrow P(t,\g_2)\in gA_f\lb t\rb.
\end{equation}
Note that the Theorem then follows from the Claim. Indeed, if $Q \in \LL[t,z]$ then there exists $g\in A$ such that $P=g \, Q \in A[t,z]$. In particular, $Q(t,\gamma_1(t)) \in A\lb t\rb $ implies that $P(t,\gamma_1(t)) \in  g A\lb t\rb$, so the Claim implies that $P(t,z) \in gA_f\lb t\rb$ and, therefore, $Q \in A_f\lb t\rb$. In order to prove the Claim, we start by noting that, since $A$ is a UFD, it is enough to prove the Claim for every irreducible element $g$ of $A$. By replacing $P$ by its remainder under its Euclidean division by $\Gamma$, furthermore, we may assume that $\deg_z(P)<d$. So let's consider the set
 $$
 \mathcal G:=\{g\in A \text{ prime }\mid \exists P,\ P(t,\g_1)\in gA\lb t\rb, P(t,z)\notin gA[t,z]\text{ and }\deg_z(P)<d\},
 $$
and let's prove that it is finite (up to multiplication by a unit). Indeed, note that if $P(t,\g_1)\in gA\lb t\rb$ and $P(t,z)\notin gA[t,z]$, we have that $\ovl{\Gamma}$ is not irreducible $A/g[t,z]$, where $\ovl R$ denote the image of a polynomial $R\in A\lb t\rb[z]$ in $A/g\lb t\rb[z]$. Thus 
$$
\prod_{i\in E_g}(z-\ovl \g_i)\in A/g[t,z]
$$
 for some $E_g\subsetneq \{1,\ldots, d\}$. For $E\subsetneq \{1,\ldots, d\}$, we set $\Gamma_E:=\prod_{i\in E}(z-\g_i)$.
 
Now assume by contradiction that $\mathcal{G}$ is infinite. In this case, there is $E\subsetneq\{1,\ldots, d\}$ such that $\Gamma_E \in A/g[t,z]$ for infinitely many primes $g$. But the coefficients of $\Gamma_E$ are in $A\lb t\rb$, and at least one of them is not in $A[t]$ because $\Gamma$ is irreducible in $A[t,z]$. Since the $\g_i$ are algebraic over $A[t]$, the coefficients of $Q_E$ are also algebraic over $A[t]$. We therefore obtain a contradiction with Lemma \ref{finite_reduction}, and conclude that $\mathcal{G}$ is finite. We may therefore define
\[
f=\prod_{g\in\mathcal G} g.
\]
Note that the Claim is verified with this choice of $f$ for every irreducible $g$ by construction. Thus the Claim is proved, finishing the case that $\pc \in \LL$.

\medskip

Now we assume that $\mathfrak c\notin \mathbb L$. Since $\pc$ is algebraic, we may write $a_0\mathfrak c^e+a_1\mathfrak c^{e-1}+\cdots+ a_e=0$ where $a_i \in A$ for every $i=0,\ldots,e$. By replacing $A$ by $A_{a_0}$ we may assume that $\mathfrak c$ is finite over $A$ (of degree $e$). For every $i$ we can write in a unique way
 \begin{equation}\label{eq_basis}
 \g_i=\g_{i,0}+\g_{i,1}\mathfrak c+\cdots+\g_{i,e-1}\mathfrak c^{e-1}
 \end{equation}
 where the $\g_{i,j}$ belong to $A\lb t\rb$ and are algebraic over $A[t]$. For $Q\in \mathbb L[t,z]$ we can expand in a unique way
 $$
 Q(t,z_0+z_1\mathfrak c+\cdots+z_{e-1}\mathfrak c^{e-1})=\sum_{k=0}^{e-1}Q_k(t, z_0,\ldots, z_{e-1})\mathfrak c^k
 $$
 where the $Q_k$ belong to $\mathbb L[t,z_0,\ldots, z_{e-1}]$. For $i\neq 1$, $\g_i$ is obtained from $\g_1$ expanded as in \eqref{eq_basis} by replacing the $\g_{1,j}$ by its conjugates $\g_{i,j}$. Following the same logic as of the first case, we are reduced to proving the claim that there is $f\in A$ such that for every $ P\in A[t,z_0,\ldots, z_{e-1}]$ and every  $g\in A$,
$$
P(t,\g_{1,0},\ldots,\g_{1,e-1})\in gA\lb t\rb\Longrightarrow P(t,\g_{2,0},\ldots,\g_{2,e-1})\in gA_f\lb t\rb.
$$
By the primitive element Theorem (that we can apply since $\LL[t]\lgw \LL\langle t\rangle$ is separable), we have that 
$$\mathbb L(t,\g_{i,0},\ldots, \g_{i,e-1})=\mathbb L\left(t,\sum_{k=0}^{e-1}\la_k\g_{i,k}\right)$$
for every $(\la_k)_k$ in a Zariski open dense subset $V_i$ of $\LL^e$. Therefore we may choose $(\la_k)_k\in\cap_{i=1}^{d}V_i$ and assume that for every $i=1$, \ldots, $d$
$$
\mathbb L(t,\g_{i,0},\ldots, \g_{i,e-1})=\mathbb L\left(t,\sum_{k=0}^{e-1}\la_k\g_{i,k}\right).
$$
Thus there is $\G_{i,k}\in\mathbb L(t)[U]$ such that
\begin{equation}\label{2nd}\g_{i,k}=\G_{i,k}\left(t,\sum_{k=0}^{e-1}\la_k\g_{i,k}\right).\end{equation}
By replacing the $\g_{i,k}$ by their conjugates $\g_{i',k}$ in \eqref{2nd} we see that we can choose the $\G_{i,k}$ to be independent of $i$. From now we denote $\G_{i,k}$ by $\G_k$, and $\sum_{k=0}^{e-1}\la_k\g_{i,k}$ by $\d_i$. By the claim made in the first case where we assumed that $\pc\in\LL$, there exists $f\in A$ such that
$$
\forall P'\in A[t,z],\forall g\in A,\ P'(t,\d_1)\in gA\lb t\rb\Longrightarrow P'(t,\d_2)\in gA_f\lb t\rb.
$$
Now, let $P\in A[t,z_0,\ldots, z_{e-1}]$ and $g\in A$ such that
$$
P(t,\g_{1,0},\ldots, \g_{1,e-1})\in gA_f\lb t\rb.
$$
Let $D(t)\in A[t]$ be a common denominator of the $\G_k$, that is, a polynomial such that $D(t)\G_k\in A[t,U]$ for every $k$.
Then there is an integer $\ell$, depending on $P$, such that
$$
R(t,z):=D(t)^\ell P(t, \G_0(t, z),\ldots, \G_{e-1}(t,z))\in A[t,z].
$$
By assumption $R(t,\d_1)\in gA\lb t\rb$, whence $R(t,\d_2)\in gA_f\lb t\rb$. We can write $D(t)=f_0t^{d_0}\times u(t)$ where $f_0\in A$ and $u(t)\in A_{f_0}[t]$ satisfies $u(0)=1$.
This shows that $P(t,\g_{2,0},\ldots, \g_{2,e-1})\in gA_{ff_0}\lb t\rb$ proving the Claim, and the Theorem is proven.
 \end{proof}

\section{On rank Theorems for non W-temperate families}\label{ssec:Quasianalytic}

We provide examples of a families of local rings where the rank Theorem does not hold. We consider an example given in \cite[Example 1.8]{BBmon}, which is based on a construction due to Nazarov, Sodin and Volberg \cite[\S\,5.3]{NSV}. We refer the reader to \cite[$\S$3]{BBmon} for a detailed presentation of quasianalytic classes, and we follow its notation.  Consider quasianalytic Denjoy-Carleman classes $\mathcal{Q}_M$ which satisfies two properties:
\begin{itemize}
\item[1)] There is a function $g \in \mathcal{Q}_M([0,1))$ which admits no extension to a function in $\mathcal{Q}_{M'}((-\delta,1))$, for all $\delta>0$ and all quasianalytic Denjoy-Carleman class $\mathcal{Q}_{M'}$ (these classes exist by \cite[\S\,5.3]{NSV});
\item[2)] The shifted class $\mathcal{Q}_{M^{(p)}}$, where $M^{(p)}_k := M_{pk}$, is a quasianalytic Denjoy-Carleman class for every $p \in \mathbb{N}$. 
\end{itemize}
For example, the class $\mathcal{Q}_M$ given by the sequence $M=(M_k)_{k\in \mathbb{N}}$, where $M_k = (\log(\log k))^k$, satisfies both conditions.

Let $\Phi: (-1,1)\lgw\R^2$ denote the $\mathcal{Q}_M$-morphism $ \Phi(u) = (u^2,g(u^2))$, and let $\phi = \Phi^{\ast}$ denote its pull-back at $0$. Note that $\Gr(\phi)= \Fr(\phi)$ since $G(x) = x_2 - \hat{g}(x_1)$ is a formal power series such that $\hat{\phi}(G) = \hat{g}(u^2) - \hat{g}(u^2) \equiv 0$. Now, suppose by contradiction that there exists a function germ $h \in \mathcal{Q}_M(-\epsilon,\epsilon)$, for some $\epsilon<1$, such that $\phi(h) = h \circ \phi(u) \equiv 0$. We remark that $\hat{h}( t, \hat{g}(t)) \equiv 0$ since
\[
0 \equiv \hat{\phi}( \hat{h}) = \hat{h}(u^2, \hat{g}(u^2)),
\]
so we conclude that the equation $h(x_1,x_2) = 0$ admits a formal solution $x_2 = \hat{g}(x_1)$. By \cite[Theorem 1.1]{BBI}, apart from shrinking $\epsilon$, there exists a function $f \in \mathcal{Q}_{M^{(p)}}(-\epsilon,\epsilon)$, for some $p\in \mathbb{N}$, such that $h(x_1,f(x_1))\equiv 0$ and $\hat{f} = \hat{g}$. Since $\mathcal{Q}_{M^{(p)}}$ is quasianalytic by condition 2) and contains $\mathcal{Q}_{M}$, we conclude that $f_{|[0,\epsilon)} = g_{|[0,\epsilon)}$. This implies that $g$ admits an extension in the shifted quasianalytic class $\mathcal{Q}_{M^{(p)}}$, contradicting condition 1). We conclude Theorem \ref{thm:TemperateRank} does not hold for these classes.

We can also consider the $o$-minimal structure $\R_{\mathcal{Q}_M}$ given by expansion of the real field by restricted functions of class $\mathcal{Q}_M$ satisfying conditions 1) and 2) above, cf.\ \cite{RSW}, and the quasianalytic class $\mathcal{Q}$ of $\mathcal{C}^\infty$ functions that are locally definable in $\R_{\mathcal{Q}_M}$. By
\cite[Theorem\,1.6]{BBC}, any function $h\in\mathcal{Q}((-1,1))$ belongs to a shifted Denjoy-Carleman class {$\mathcal{Q}_{M^{(p)}}$}, for some positive integer $p$. We conclude that the morphism $\Phi$ defined above shows that the rank Theorem can not hold for $\mathcal{Q}$. 

\begin{remark}
Every quasianalytic class which properly contains the analytic functions does not satisfy the Weierstrass preparation property \cite{PR2}. 
\end{remark}

Finally, let us note that Theorem \ref{thm:TemperateRank} holds for at least one family of rings which is not Weierstrass. For instance, if we set $\KK\nl x_1,\ldots, x_n\nr=\KK[x_1,\ldots, x_n]_{(x_1,\ldots, x_n)}$ for any integer $n$, then for any morphism $\phi:\KK\nl \x \nr\lgw \KK\nl \ub \nr$, we have
$$
\Gr(\phi)=\Fr(\phi)=\dim\left(\KK\nl \x \nr/\Ker(\phi)\right)
$$
essentially by Chevalley's constructible set Theorem. But the rings of rational functions are not Henselian local rings, so they do not form a Weierstrass family, c.f. Proposition \ref{rk:Wtemp}\,\ref{prop:hensel}.



\end{document}